\documentclass[11pt]{article}
\usepackage[english]{babel}
\usepackage{geometry}
\usepackage{graphicx} 
\usepackage{amsmath,amsthm,amssymb,amscd,amsfonts}
\usepackage{tikz-cd}
\usepackage[all,cmtip]{xy}
\usepackage{mathrsfs}
\usepackage{pdfpages}
\usepackage{float}
\usepackage{caption,subcaption}
\usepackage{xcolor}
\usepackage{suffix}
\usepackage{mathtools}
\usepackage{soul}
\usepackage{gensymb}
\usepackage{enumitem}
\usepackage{calc}
\usepackage{csquotes}
\graphicspath{ {images/} }
\usepackage{gensymb}
\usepackage{placeins}
\usepackage[toc,page]{appendix}
\usepackage{verbatim}
\usepackage[nottoc]{tocbibind}
\usepackage{textcomp}
\usepackage{listings}
\usepackage{stmaryrd}
\usepackage{url}
\usepackage{xargs}
\usepackage{comment}
\usepackage{lipsum}
\usepackage[colorlinks = true, linkcolor = black, urlcolor  = black, citecolor = black, anchorcolor = black]{hyperref}
\usepackage[noabbrev,nameinlink]{cleveref} 

\theoremstyle{plain}
\newtheorem{theorem}{Theorem}
\newtheorem{lemma}{Lemma}

\newtheorem{proposition}{Proposition}

 \theoremstyle{definition}
\newtheorem{definition}{Definition}

\newtheorem{example}{Example}

\crefname{section}{\S}{\S\S}
\crefname{appendix}{App.}{Apps.}

\Crefname{equation}{Eq.}{Eqs.}
\Crefname{theorem}{Thm.}{Thms.}
\Crefname{lemma}{Lem.}{Lems.}
\Crefname{corollary}{Cor.}{Cors.}
\Crefname{proposition}{Prop.}{Props.}

\Crefname{definition}{Def.}{Defs.}
\Crefname{notation}{Not.}{Nots.}
\Crefname{conjecture}{Conj.}{Conjs.}
\Crefname{remark}{Rk.}{Rks.}
\Crefname{problem}{Prob.}{Probs.}
\Crefname{example}{Ex.}{Exs.}

\Crefname{figure}{Fig.}{Figs.}
\Crefname{tabular}{Tab.}{Tabs.}

\DeclareMathOperator\spn{span} 

\DeclareMathOperator\Ad{Ad} 

\DeclareMathOperator\SO{SO}

\catcode`,\active

\catcode`\,12

\catcode`,\active

\catcode`\,12

\DeclarePairedDelimiterX\MeijerM[3]{\lparen}{\rparen}%
{\begin{smallmatrix}#1 \\ #2\end{smallmatrix}\delimsize\vert\,#3}

\newcommand\MeijerG[8][]{%
  G^{\,#2,#3}_{#4,#5}\MeijerM[#1]{#6}{#7}{#8}}

\WithSuffix\newcommand\MeijerG*[7]{%
  G^{\,#1,#2}_{#3,#4}\MeijerM*{#5}{#6}{#7}}

\newcommand{\iprod}{\mathbin{\lrcorner}} 

\SetSymbolFont{stmry}{bold}{U}{stmry}{m}{n}

\parskip=\medskipamount
\usepackage[normalem]{ulem}

\begin{document}

\title{Exterior differential systems on Lie algebroids  and the invariant inverse problem of the calculus of variations}

\author{Tom Mestdag and Kenzo Yasaka\\[2mm]
	{\small Department of Mathematics,  University of Antwerp,}\\
	{\small Middelheimlaan 1, 2020 Antwerpen, Belgium}
}

\date{}

\maketitle

\begin{abstract}
We extend the theory of exterior differential systems from manifolds and their tangent bundles to Lie algebroids. In particular, we define the concept of an integral manifold of such an exterior differential system. We support our developments with several examples, including an application to dynamical systems with a symmetry group and to the invariant inverse problem of the calculus of variations.
\end{abstract}

\vspace{3mm}

\vspace{3mm}






\maketitle



\section{Introduction}
Roughly speaking, an \emph{exterior differential system (EDS)} is a system of differential equations on a manifold defined by setting a certain set of differential forms equal to zero. Of particular interest are \emph{integral manifolds}, which are submanifolds satisfying all equations of the system. Notably, a system of partial differential equations, with any number of independent and dependent variables and involving partial derivatives of any order, can be written as an EDS. Therefore, integral manifolds provide a geometric description of finding solutions to a system of PDEs.

The theory of EDS originates from the ideas of several mathematicians, such as J.F.\ Pfaff, G.\ Darboux \cite{darboux1882pfaff}, {\'E}. Goursat and, most prominently, {\'E}.\ Cartan \cite{cartan1922invariant,cartan1945eds}. E.\ K{\"a}hler also contributed much to the field \cite{kahler1934ckthm}, which resulted in an existence theorem which we now know as the \emph{Cartan--K{\"a}hler theorem}. EDS has proven useful in the study of various mathematical problems, such as the isometric embedding problem \cite{bryant1991exterior}, Cartan's equivalence method \cite{olver1995equivalence} and the inverse problem of the calculus of variations \cite{anderson1992invprob}.

One special subcase of this last problem is part of the motivation of the current paper. In general, the inverse problem deals with the question of whether a regular Lagrangian exists for a given set of second-order differential equations. One way of answering this question lies in finding a so-called \emph{multiplier} that solves a set of necessary and sufficient conditions, the so-called Helmholtz conditions. Since this mainly involves solving PDEs, many recent attempts make use of EDS on manifolds \cite{aldridge2003invprob,aldridge2,anderson1992invprob,do2016invprob,doprince}. However, up to now, there does not yet exist a complete solution to the problem, apart from dimension 2 \cite{Douglas}. In this paper, we will address a particular subproblem: \emph{Does a $G$-invariant Lagrangian exist for a given $G$-invariant second-order system on a Lie group $G$?} This is the so-called \emph{invariant inverse problem}  (see e.g., \cite{crampin2008invariantinv,muzsnay2005invariantinv}). We will show that, if one wants to make full use of the symmetry in the problem, one needs to devise a new technique for solving it, that of \emph{EDS on Lie algebroids}.

In the first half of the paper, we generalize the theory of EDS from manifolds to Lie algebroids. Our results are more general than those in the literature, in the sense that they reduce to those in the specific case where the Lie algebroid is taken to be the tangent bundle of a manifold (cf. \cite{bryant1991exterior,ivey2016cartan}). In the second half of the paper, we motivate the need for our extension by showing how three seemingly disparate problems involving differential equations can all be recast as an EDS on a certain Lie algebroid.

In \Cref{sec:preliminaries}, we present the preliminaries, in an attempt to keep the paper self-contained. In particular, we give an introduction to Lie algebroids and define an exterior derivative operator on sections of the dual of the Lie algebroid in order to meaningfully define an EDS on a Lie algebroid. We will also discuss the geometry of the space $\mathbb{R}\times \mathrm{T}M$ of a smooth manifold $M$. Such geometry is useful in the study of the inverse problem (cf. \cite{crampin1984geohelmholtz,krupkova2007jetbundles}).

In \Cref{sec:pullback}, we discuss prolongation bundles and their Lie algebroid structure (as introduced in e.g., \cite{Higgins,martinez2001lagrangeliealgebroid, martinez2002prolongation}, see \Cref{def:prol}), which will allow us to define the concept of an integral manifold for a set of forms on a Lie algebroid. The intuition for why this type of construction is necessary is as follows. If we are given a Lie algebroid $A\to M$, we would want integral manifolds to be seen as the image of a section of a fiber bundle $p\colon \overline{M}\to M$, in such a way that it represents a solution of a certain system of PDEs, where we let $\overline{M}$ consist of all coordinates in $M$, together with the dependent variables of this system. The $p$-prolongation bundle of $A$ is a new Lie algebroid $\overline{A}\to \overline{M}$, which represents this system of PDEs.  In \Cref{newprop1} we show that we can split the prolongation bundle, by making use of a connection on $p$. In \Cref{sec:EDSintman}, we define integral manifolds for Lie algebroids, and we give an example where we consider a differential ideal (a graded ideal in the ring of forms on a Lie algebroid closed under the exterior derivative operator) generated only by $1$-forms. The efficacy of our definition is confirmed by the result in \Cref{prop:depcond}.

As a first concrete example of a problem which we can reformulate as an EDS on a specific Lie algebroid, we consider a first-order semilinear PDE. We will show in \Cref{sec:applications} that solving this PDE is equivalent to finding integral manifolds of a differential ideal generated by a single $1$-form. 

Besides being an extension of the standard definition for a tangent bundle, we show in \Cref{sec:symmetry} that our definition also applies to Atiyah Lie algebroids, which are central to the study of  differential equations with a  symmetry Lie group (see, for example, \cite{barbero2016invproblemliealgebroids, firstorderrec,deleon2005survey,Educontrol}).  In \Cref{prop4}, we show first that the prolongation of an Atiyah Lie algebroid can again be identified with an Atiyah Lie algebroid. Our second application concerns finding solutions of first-order ODEs with symmetry.  We refer to e.g., \cite{fels2007edssymmetry,fels2008edssymmetry} for another approach, using EDS on manifolds.

As announced before, in \Cref{sec:invariantinverse}, we present our most substantial application. To the best of our knowledge, using exterior differential systems (on either manifolds or Lie algebroids) to solve the invariant inverse problem  on a Lie group has never been attempted. It was shown in \cite{crampin2008invariantinv}, that such a Lagrangian exists if and only if two conditions are satisfied. The first condition is the existence of a \emph{reduced multiplier matrix} for the system reduced to the level of its Lie algebra (\Cref{def:reducedmultiplier}). The second condition is a cohomological condition. In \Cref{prop:cohomology} we first establish a time-dependent version of the cohomological  condition. Then, we show in \Cref{thm:twoform} how to rewrite the first condition in terms of a two-form on a specific Lie algebroid, that we call the \emph{IP Lie algebroid} in this context (\Cref{def:IPalgebroid}). In \Cref{prop:new} we describe how the search for a reduced multiplier can be formulated as an EDS problem  on the IP Lie algebroid. Finally, in \Cref{thm:helmholtzintmfd} we establish the explicit link between reduced multipliers and integral manifolds, in the case that the algebraic ideal at hand is already a differential ideal.  

We conclude the paper with some suggestions for future work.

\section{Preliminaries}\label{sec:preliminaries}
\subsection{Lie algebroids}

We start with the definition of a Lie algebroid. We refer to e.g., \cite{mackenzie1987algebroid,marle2008algebroid} for more details about the concepts we introduce in this subsection. In the following, we denote the space of sections of a vector bundle $A\to M$ by $\Gamma(A)$. In particular, if $A$ is the tangent bundle of a manifold $M$, we denote the collection of smooth vector fields by $\mathcal{X}(M)\coloneqq \Gamma(\mathrm{T}M)$, and its usual Lie bracket by $[\cdot,\cdot]$.

\begin{definition}
A \emph{Lie algebroid} is a triple $(A,\llbracket \cdot,\cdot\rrbracket,\rho)$, where
\begin{itemize}
    \item $\tau\colon A\to M$, $(x^i,w^{\alpha})\mapsto (x^i)$ is a vector bundle over a manifold $M$,
    \item $\rho\colon A\to \mathrm{T}M$ is a linear bundle map,
    \item $(\Gamma(A),\llbracket \cdot,\cdot\rrbracket)$ is a Lie algebra,
    \item $\llbracket \sigma_1,f\sigma_2\rrbracket=\rho(\sigma_1)(f)\sigma_2+f\llbracket \sigma_1,\sigma_2\rrbracket$, where $\sigma_1,\sigma_2\in \Gamma(A)$, $f\in C^{\infty}(M)$. Here it is used that the map $\rho$ descends to a map between sections $\rho\colon \Gamma(A)\to \mathcal{X}(M)$ (also denoted by $\rho$).
\end{itemize}
We call $\rho$ the \emph{anchor map}, and $\llbracket \cdot,\cdot\rrbracket$ the \emph{bracket} of $A$.
\end{definition}


It is well known that Lie algebroids generalize two fundamental concepts: The tangent bundle $\mathrm{T}M$ of a manifold $M$ (with the Lie bracket of vector fields and $\rho=\mathrm{id}$) and Lie algebras (which are Lie algebroids over a point with $\rho=0$).

Locally, if $\{e_{\alpha}\}$ is a basis of sections of $A$, one can write
\begin{equation*}
    \rho(e_{\alpha})=\rho_{\alpha}^i(x)\frac{\partial}{\partial x^i}\in \mathcal{X}(M),
\end{equation*}
\begin{equation*}
    \llbracket e_{\alpha},e_{\beta}\rrbracket=L_{\alpha \beta}^{\gamma}(x)e_{\gamma}.
\end{equation*}
It is easy to see that $L_{\alpha \beta}^{\gamma}=-L_{\beta \alpha}^{\gamma}$. The real local functions $\{L_{\alpha \beta}^\gamma\}$ are called the \emph{structure functions}. Since the map $\rho\colon \Gamma(A)\to \mathcal{X}(M)$ is a Lie algebra morphism, it follows from $\rho(\llbracket e_{\alpha},e_{\beta}\rrbracket)=[\rho(e_{\alpha}),\rho(e_{\beta})]$ that the following identity always holds:
\begin{equation}\label{eqn:anchorspecialidentity}
    \rho_{\gamma}^k L_{\alpha \beta}^{\gamma}=\rho_{\alpha}^i \frac{\partial \rho_{\beta}^k}{\partial x^i}-\rho_{\beta}^j \frac{\partial \rho_{\alpha}^k}{\partial x^j}.
\end{equation}

We now construct the induced exterior derivative operator on sections of exterior powers of the dual of the Lie algebroid (which we will call a \emph{form} on a Lie algebroid). We will denote the exterior derivative operator by $\delta$, and $q$-forms on $A\to M$ by $\Lambda^q(A)$.

\begin{definition}
The map
\begin{equation*}
    \delta\colon \Lambda^q(A)\longrightarrow \Lambda^{q+1}(A),\quad \omega \longmapsto \delta \omega
\end{equation*}
for $q=0$ is defined by
\begin{equation*}
    \delta f(\sigma)=\rho(\sigma) f
\end{equation*}
for any $\sigma\in \Gamma(A)$ and $f\in \Lambda^0(A)=C^{\infty}(M)$,  and for $q\geq 1$ by
\begin{multline*}
    \delta \omega(\sigma_1,\dots,\sigma_{q+1})=\sum_{i=1}^{q+1} (-1)^{i+1}\rho(\sigma_i)\left(\omega(\sigma_1,\dots,\hat{\sigma}_i,\dots,\sigma_{q+1})\right)\\+\sum_{1\leq i<j\leq q+1} (-1)^{i+j} \omega(\llbracket \sigma_i,\sigma_j\rrbracket,\sigma_1,\dots,\hat{\sigma}_i,\dots,\hat{\sigma}_j,\dots,\sigma_{q+1}),
\end{multline*}
for any $\sigma_1,\dots,\sigma_{q+1}\in \Gamma(A)$ and $\omega\in \Lambda^q(A)$. In the above expression, a hat denotes the omission of that element.
\end{definition}
It is not difficult to verify that $\delta^2=0$ and that for any $\omega\in \Lambda^q(A)$ and $\theta\in \Lambda^r(A)$,
\begin{equation*}
    \delta(\omega\wedge \theta)=\delta \omega\wedge \theta+(-1)^q \omega\wedge \delta \theta.
\end{equation*}
If $\{e^{\alpha}\}$ is the dual basis of $\{e_{\alpha}\}$, then $\delta$ is given by
\begin{equation}\label{eqn:structure}
    \delta f=\frac{\partial f}{\partial x^i}\rho_{\alpha}^i e^{\alpha},\quad \delta e^{\alpha}=-\frac{1}{2}L_{\beta \gamma}^{\alpha}(x) e^{\beta}\wedge e^{\gamma}.
\end{equation}

We will also need to define two other operations on forms.
\begin{definition}
The \emph{Lie derivative} of a form with respect to a section $\sigma\in \Gamma(A)$ is a map
\begin{equation*}
    \mathcal{L}_{\sigma}\colon \Lambda^q(A)\longrightarrow \Lambda^q(A),
\end{equation*}
given for $q=0$ by
\begin{equation*}
    \mathcal{L}_{\sigma}(f)=\rho(\sigma)(f),
\end{equation*}
for any $f\in \Lambda^0(A)=C^{\infty}(M)$ and for $q\geq 1$ by
\begin{equation*}
    \mathcal{L}_{\sigma} \omega(\sigma_1,\dots,\sigma_q)=\rho(\sigma)(\omega(\sigma_1,\dots,\sigma_q))-\sum_{i=1}^q \omega(\sigma_1,\dots,\llbracket \sigma,\sigma_i\rrbracket,\dots,\sigma_q),
\end{equation*}
for any $\sigma_1,\dots,\sigma_q\in \Gamma(A)$ and $\omega\in \Lambda^q(A)$.
\end{definition}

The Lie derivative satisfies a \emph{Leibniz rule}
\begin{equation*}
    \mathcal{L}_{\sigma}(\omega\wedge \theta)=(\mathcal{L}_{\sigma}\omega)\wedge \theta+\omega\wedge (\mathcal{L}_{\sigma}\theta),
\end{equation*}
and a \emph{Cartan formula}, namely
\begin{equation*}
    \mathcal{L}_{\sigma}=\delta\circ \iota_{\sigma}+\iota_{\sigma}\circ \delta,
\end{equation*}
where $\sigma\in \Gamma(A)$, $\omega\in \Lambda^q(A)$ and $\theta\in \Lambda^r(A)$. Here $\iota_{\sigma}\colon \Lambda^q(A)\to \Lambda^{q-1}(A)$ is the \emph{interior product} of a form by $\sigma$, i.e.,
\begin{equation*}
    \iota_{\sigma}\omega(\sigma_1,\dots,\sigma_{q-1})=\omega(\sigma,\sigma_1,\dots,\sigma_{q-1}),
\end{equation*}
for all $\sigma_1,\dots,\sigma_{q-1}\in \Gamma(A)$.

\subsection{\texorpdfstring{Geometry of the evolution space $\mathbb{R}\times \mathrm{T}M$}{Geometry of the evolution space R x TM}}\label{subsec:evolution}

In this subsection, we introduce the basics on the geometry of the so-called \emph{evolution space} $E\coloneqq \mathbb{R}\times \mathrm{T}M$, where $M$ is a smooth manifold. For a more in-depth treatment on this topic, one can refer to \cite{crampin1984geohelmholtz,do2016invprob,krupkova2007jetbundles}.

Suppose that we are given a vector field $X\in \mathcal{X}(M)$ on a smooth manifold $M$, given in coordinates by
\begin{equation*}
    X=X^i \frac{\partial}{\partial x^i}.
\end{equation*}
Any such vector field can be lifted to a vector field on $\mathrm{T}M$, by means of the \emph{complete lift} $X^C$ and \emph{vertical lift} $X^V$, given in canonical coordinates $(x^i,\dot{x}^i)$ by
\begin{equation*}
    X^C\coloneqq X^i \frac{\partial}{\partial x^i}+\frac{\partial X^i}{\partial x^j}\dot{x}^j\frac{\partial}{\partial \dot{x}^i},
\end{equation*}
\begin{equation*}
    X^V\coloneqq X^i V_i,\quad V_i\coloneqq \frac{\partial}{\partial \dot{x}^i}.
\end{equation*}
Some convenient formulae for the brackets of complete and vertical lifts are given by
\begin{equation*}
    [X^C,Y^C]=[X,Y]^C,\quad [X^C,Y^V]=[X,Y]^V,\quad [X^V,Y^V]=0,
\end{equation*}
where $X,Y\in \mathcal{X}(M)$. If we have a second-order differential equation field $\Gamma$ at our disposal, say
\begin{equation*}
    \Gamma=\frac{\partial}{\partial t}+\dot{x}^i \frac{\partial}{\partial x^i}+F^i \frac{\partial}{\partial \dot{x}^i},
\end{equation*}
which represents the differential equations $\ddot{x}^i=F^i(t,x^j,\dot{x}^j)$, we can also define the \emph{horizontal lift} of $X\in \mathcal{X}(M)$, given by
\begin{equation*}
    X^H\coloneqq X^i H_i,
\end{equation*}
where
\begin{equation*}
    H_i\coloneqq \frac{\partial}{\partial x^i}-\Gamma_i^j\frac{\partial}{\partial \dot{x}^j},\quad \Gamma_i^j\coloneqq -\frac{1}{2}\frac{\partial F^j}{\partial \dot{x}^i}.
\end{equation*}
We can also lift the vector fields $X^V, X^H$ on $\mathrm{T}M$ to $E$, in an obvious way. A basis of $\mathcal{X}(E)$ is then given by $\{\Gamma,V_i,H_i\}$. Its dual basis is given by $\{\mathrm{d}t,\psi^i,\theta^i\}$, where
\begin{equation*}
    \psi^i=\mathrm{d}\dot{x}^i-F^i\, \mathrm{d}t+\Gamma_j^i \theta^j,\quad \theta^i=\mathrm{d}x^i-F^i\, \mathrm{d}t.
\end{equation*}
Calculating Lie brackets on these vector fields gives the following expressions (see also \cite{crampin1984geohelmholtz,do2016invprob,krupkova2007jetbundles}):
\begin{equation} \label{prop:bracketsstandard1}
    [\Gamma,V_i]=-H_i+\Gamma_i^j H_j,\quad [\Gamma,H_i]=\Gamma_i^j H_j+\Phi_i^j V_j,\quad [V_i,V_j]=0,
\end{equation}
\begin{equation}\label{prop:bracketsstandard2}
    [H_i,V_j]=\Gamma_{ij}^k V_k,\quad [H_i,H_j]=R_{ij}^k V_k,
\end{equation}
where
\begin{equation*}
    \Phi_i^j=-\frac{\partial F^j}{\partial x^i}-\Gamma_k^j \Gamma_i^k-\Gamma(\Gamma_i^j),\quad \Gamma_{ij}^k=\frac{\partial \Gamma_i^k}{\partial \dot{x}^j}=-\frac{1}{2}\frac{\partial F^k}{\partial \dot{x}^i \partial \dot{x}^j},
\end{equation*}
\begin{equation*}
    R_{ij}^k=-\frac{\partial \Gamma_j^k}{\partial x^i}+\frac{\partial \Gamma_i^k}{\partial x^j}+\Gamma_i^l\Gamma_{lj}^k-\Gamma_j^l \Gamma_{li}^k.
\end{equation*}
In particular, $\Gamma_{ij}^k=\Gamma_{ji}^k$ and $R_{ij}^k=-R_{ji}^k$.



\section{The prolongation Lie algebroid}\label{sec:pullback}

In this section we assume that we have a second fiber bundle at our disposal, and we show that we can extend a given Lie algebroid to one over the total manifold of that fiber bundle. This construction lies at the basis of the concepts we will introduce in the next section.

Let $p\colon \overline{M}\to M$ be a fiber bundle and let $\tau\colon A\to M$ be a Lie algebroid with anchor $\rho$ and bracket $\llbracket \cdot,\cdot\rrbracket^A$.  We consider the manifold
\begin{equation*}
    \overline{A}\coloneqq \rho^*\mathrm{T}\overline M  = \{ ( a_m, {\bar v}_{\bar m})\in  A\times \mathrm{T}\overline{M} \mid \rho(a_m) = \mathrm{T}p({\bar v}_{\bar m})\}.
\end{equation*}
By construction $p(\bar m) = m$.
 
The projection $\bar\tau\colon \overline{A}\to \overline{M}, (a_m, {\bar v}_{\bar m}) \mapsto \bar m$ gives it the structure of a vector bundle. This bundle is called \emph{the $p$-prolongation bundle of $A$}. It is well known (see \cite{deleon2005survey,martinez2002prolongation} for details) that the prolongation bundle is again a Lie algebroid, with the second projection serving as the anchor map for $\overline{A}$,
\begin{equation}\label{eqn:anchorprolongation}
    \bar{\rho}\colon \overline A \longrightarrow \mathrm{T}\overline M,\quad (a_m, {\bar v}_{\bar m}) \longmapsto {\bar v}_{\bar m}.
\end{equation}
We will use $P$ for the first projection, i.e.,
\begin{equation*}
    P\colon \overline A \longrightarrow A,\quad ( a_m, {\bar v}_{\bar m}) \longmapsto  a_m.
\end{equation*}
Then, the following diagram commutes:
\begin{equation*}\label{eqn:cdtwo}
\begin{CD}
\overline{A}@>\bar\rho>>\mathrm{T}\overline M\\
@V P VV @VV\mathrm{T}p V\\
A@>>\rho>\mathrm{T}M.\\
\end{CD}
\end{equation*}
To write an expression for its Lie algebroid bracket, we first need to explain what projectable sections of $\bar{\tau}$ are. In general a section of $\bar{\tau}$ is of the form $(\sigma,X)$, where $\sigma$ is a section of $\tau$ \emph{along $p$} and $X$ is a vector field on $\overline M$. If $\sigma$ is in fact a section of $\tau$ (i.e., a `basic' section of $\tau$ along $p$), we say that $(\sigma,X)$ is \emph{projectable}. To define a Lie algebroid bracket on $\overline A$ it is enough to state how it acts on two projectable sections, and then extend it to all sections by imposing the Leibniz identity. Here, we can set
\begin{equation}\label{eq:defprolbracket}
    \llbracket(\sigma_1,X_1),(\sigma_2,X_2)\rrbracket (\bar m)\coloneqq \left(\llbracket\sigma_1,\sigma_2\rrbracket^A(m), [X_1,X_2](\bar m)\right).
\end{equation}

\begin{definition} \label{def:prol}
We call $\overline A = \rho^*\mathrm{T}\overline M\to \overline{M}$, together with its anchor \eqref{eqn:anchorprolongation} and its bracket defined on projectable sections by \eqref{eq:defprolbracket} the \emph{$p$-prolongation Lie algebroid of $A$}.
\end{definition}

We denote the local expression for the projection $p\colon \overline{M}\to M$ by $(x^i,y^{\mu})\mapsto (x^i)$. We will assume that $\{e_{\alpha}\}$ is a local basis of sections of $A$, and that fiber coordinates $w^\alpha$ with respect to this basis are such that $\tau\colon A\to M$, $(x^i,w^{\alpha})\mapsto (x^i)$. Then, the following are (projectable) sections of $\overline{A}$: 
\begin{equation*} \label{eq:basis1}
    \bar{e}_{\alpha} \colon (x^i,y^{\mu})\longmapsto \left(x^i,y^{\mu},e_{\alpha}(x),\rho^i_\alpha \left.\frac{\partial}{\partial x^i}\right|_{(x,y)}\right),
\end{equation*}
\begin{equation*} \label{eq:basis2}
    \overline{E}_{\mu}\colon (x^i,y^{\mu})\longmapsto \left(x^i,y^{\mu},0,\frac{\partial}{\partial y^{\mu}}\Big|_{(x,y)}\right).
\end{equation*}
Together, the sections $\{\bar{e}_{\alpha},\overline{E}_{\mu}\}$ form a local basis of sections of $\overline{A}$. Any vector ${\bar a}_{\bar m}$ in ${\overline A}_{\bar m}$ can be decomposed as
\begin{equation*}
    {\bar a}_{\bar m} = w^\alpha {\bar e}_\alpha(\bar m)+ v^\mu {\overline{E}_\mu}(\bar m).
\end{equation*}
With the help of these fiber coordinates, we may write the projection $\bar\tau\colon \overline A \to \overline M$ locally as
\begin{equation*}
    (x^i,y^{\mu},w^{\alpha},v^{\mu})\longmapsto (x^i,y^{\mu}).
\end{equation*}
The Lie algebroid bracket relations of the basis are
\begin{equation*}\label{eqn:brackets}
    \llbracket \bar{e}_{\alpha},\bar{e}_{\beta}\rrbracket=L_{\alpha\beta}^\gamma {\bar e}_\gamma,\qquad \llbracket \bar{e}_{\alpha},\overline{E}_{\mu}\rrbracket=0,\qquad \llbracket \overline{E}_{\mu},\overline{E}_{\nu}\rrbracket=0,
\end{equation*}
and the anchor map is 
\begin{equation*} 
    \bar{\rho}(\bar{e}_{\alpha}) = \rho^i_\alpha \frac{\partial}{\partial x^i},\qquad \bar{\rho}(\overline{E}_{\mu})=\frac{\partial}{\partial y^{\mu}}.
\end{equation*}
We can also write this as
\begin{equation*}\label{defrhobar}
    \bar{\rho}\colon \overline{A}\longrightarrow \mathrm{T}\overline{M},\quad (x^i,y^{\mu},w^{\alpha},v^{\mu})\longmapsto (x^i,y^{\mu},\dot{x}^i=\rho_{\alpha}^i w^{\alpha}, v^\mu).
\end{equation*}

Let $\bar{\delta}$ denote the exterior derivative operator on the algebroid $\overline{A}$. Then a simple computation shows that for $\bar{f}\in C^{\infty}(\overline{M})$,
\begin{equation*}
    \bar{\delta} \bar{f}=\rho_{\alpha}^i \frac{\partial \bar{f}}{\partial x^i} \bar{e}^{\alpha}+ \frac{\partial \bar{f}}{\partial y^{\mu}}\overline{E}^{\mu},
\end{equation*}
and  in particular $\bar{\delta} x^i=\rho_{\alpha}^i\bar{e}^{\alpha}$ and $\bar{\delta}y^\mu=   \overline{E}^{\mu}$. If $\{\bar{e}^{\alpha},\overline{E}^{\mu}\}$ is the dual basis of $\{\bar{e}_{\alpha},\overline{E}_{\mu}\}$, then
\begin{equation*}
    \bar{\delta} \bar{e}^{\alpha}=-\frac{1}{2}L_{\beta \gamma}^{\alpha}\bar{e}^{\beta}\wedge \bar{e}^{\gamma},\quad \bar{\delta} \overline{E}^{\mu}=0.
\end{equation*}

Recall that  $P\colon \overline{A}\to A$ is the map $( a_m, {\bar v}_{\bar m}) \mapsto  a_m$. In local coordinates, it is given by  $(x^i,y^{\mu},w^{\alpha},v^{\mu})\mapsto (x^i,w^{\alpha})$. It is the projection of a fiber bundle since $\tau$ is a smooth map and $p$ is the projection of a fiber bundle. The following diagram commutes:
\begin{equation}\label{eqn:cd}
\begin{CD}
\overline{A}@>P>>A\\
@V{\bar{\tau}}VV @VV\tau V\\
\overline{M}@>>p>M.\\
\end{CD}
\end{equation}
With the help of $P$, we can identify any 1-form $\theta$ on $A$ with a unique 1-form $\bar\theta$ on $\overline A$. Indeed, for any $\bar m \in \overline M$ and $\bar\sigma\in\Gamma(\overline A)$, we can set
\begin{equation*}
    \bar\theta_{\bar m}\Big(\bar\sigma(\bar m)\Big)\coloneqq \theta_{p(\bar m)}\Big(P(\bar\sigma(\bar m))\Big).
\end{equation*}
Using this,  we may identify $e^\alpha$ and ${\bar e}^\alpha$ as either forms on $A$ or $\overline A$, respectively.

\begin{example} \label{eg:example1}
Let $A$ be the tangent bundle $\mathrm{T}M$ over $M$. Since $\rho=\rm{id}$ it is clear that $\overline A=\rho^*\mathrm{T}\overline M = \mathrm{T}\overline M$. In what follows, we will refer to this example as the ``standard case'', when compared to the existing literature on EDS on manifolds \cite{bryant1991exterior,ivey2016cartan}.

We can learn a useful fact about $\overline{A}$, in general, by examining this example in some further detail.
By local triviality of the tangent bundle, the tangent map $(\mathrm{T}p)_{\bar{m}}\colon \mathrm{T}_{\bar{m}} \overline{M}\to \mathrm{T}_{p({\bar{m}})} M$ of the projection $p$ is surjective for all ${\bar{m}}\in \overline{M}$. We therefore have a short exact sequence of vector spaces
\begin{equation*}
    \begin{CD}
    0@>>>\ker((\mathrm{T}p)_{\bar{m}})@>>>\mathrm{T}_{\bar{m}} \overline{M}@>(\mathrm{T}p)_{\bar{m}}>>\mathrm{T}_{p(\bar{m})}M@>>>0.
    \end{CD}
\end{equation*}
Now note that $\mathrm{T}p$ induces a bundle map $\widetilde{\mathrm{T}p}\colon \mathrm{T}\overline{M}\to p^{\ast} (\mathrm{T}M)$ over $\overline{M}$ via
\begin{equation*}
    v_{\bar{m}}\longmapsto (p(\bar{m}),\mathrm{T}p(v_{\bar{m}})),
\end{equation*}
where $v_{\bar{m}}\in \mathrm{T}_{\bar{m}} \overline{M}$. This map is also surjective because it maps the fiber $\mathrm{T}_{\bar{m}}\overline{M}$ onto the fiber $(p^{\ast} (\mathrm{T}M))_{\bar{m}}\simeq \mathrm{T}_{p(\bar{m})} M$. Since the kernel of this map is precisely the vertical bundle $\mathrm{V}p$, we also have a short exact sequence of vector bundles over $\overline{M}$ \cite{greub1973derham,tu2017connections}
\begin{equation*}
    \begin{CD}
    0@>>>\mathrm{V}p@>>>\mathrm{T} \overline{M}@>\widetilde{\mathrm{T}p}>>p^{\ast}(\mathrm{T}M)@>>>0.
    \end{CD}
\end{equation*}
When it is clear from context, we will from now on denote $\widetilde{\mathrm{T}p}$ by $\mathrm{T}p$. An \emph{Ehresmann connection on $p$} is a splitting $h\colon p^*\mathrm{T}M \to \mathrm{T}\overline M$ of this short exact sequence. The map $h$ is often called a \emph{horizontal lift}. Its defining property is that it satisfies $\mathrm{T}p(h(\bar m,v_m)) = v_m$. Since all short exact sequences of real vector bundles over smooth manifolds split, it follows that we may identify
\begin{equation*}
    \mathrm{T}\overline{M}\simeq p^{\ast}(\mathrm{T}M)\oplus_{\overline M} \mathrm{V}p.
\end{equation*}
\end{example}

We show below that the splitting property we've discovered in the example, after using an Ehresmann connection, also works in the case when $A$ is a general Lie algebroid. This gives us a second interpretation of the prolongation bundle. Let $p^{\ast} A\to \overline{M}$ be the pullback bundle of $\tau$ by $p$.

\begin{proposition} \label{newprop1} Using an Ehresmann connection $h$ on $p$, we can split the $p$-prolongation of $A$,
\begin{equation*}
    \overline A  =\rho^*\mathrm{T}\overline M \simeq   p^{\ast}A\oplus_{\overline{M}} \mathrm{V}p.
\end{equation*}
\end{proposition}


\begin{proof}
Consider the projection $\overline A \to p^* A, (a_m,{\bar v}_{\bar m})\mapsto (\bar m, a_m)$. By construction, a pair   lies in its kernel when it is of the form $(0,{\bar v}_{\bar m})$ with ${\bar v}_{\bar m}\in \mathrm{V}p$. In what follows we will simply write $\mathrm{V}p$ for this kernel.
We have therefore a short exact sequence of vector bundles
\begin{equation*}
    \begin{CD}
    0@>>>\mathrm{V}p@>>> \overline{A}@>>>p^{\ast}A @>>>0.
    \end{CD}
\end{equation*}
Using the Ehresmann connection $h\colon p^*\mathrm{T}M \to \mathrm{T}\overline M$, we define
\begin{equation*}
    \bar h\colon p^*A \longrightarrow \overline A,\quad (\bar m,a_m) \longmapsto (a_m, h(\bar m, \rho(a_m))).
\end{equation*}
This does indeed define a pair  in $\overline A$  because by definition $\mathrm{T}p( h(\bar m, \rho(a_m)))=\rho(a_m)$. Since $ h(\bar m, \rho(a_m))$ is a vector over $\bar m$, and since the first element of the pair is simply $a_m$, it is clear that $\bar h$ is, in fact, a splitting of the above sequence. We may therefore conclude that $\overline A \simeq  p^{\ast}A\oplus_{\overline{M}} \mathrm{V}p.$ \end{proof}

We can make the identification we have discovered in \Cref{newprop1} a bit more explicit. Let us denote $\widehat A\coloneqq p^{\ast}A\oplus_{\overline{M}} \mathrm{V}p$. If $\omega_{\bar m}\colon \mathrm{T}_{\bar m}{\overline M} \to \mathrm{V}_{\bar m}p, {\bar v}_{\bar m} \mapsto {\bar v}_{\bar m} - h(\bar m, \mathrm{T}p ({\bar v}_{\bar m}))$ is the corresponding left splitting of  $ h$, the map
\begin{equation*}
    \alpha_h\colon  \overline A \longrightarrow \widehat A,\quad (a_m, {\bar v}_{\bar m})\longmapsto  a_m\oplus \omega({\bar v}_{\bar m})  
\end{equation*}
is a bundle isomorphism with inverse
\begin{equation*}
    {\alpha}_h^{-1}\colon \widehat A \longrightarrow \overline A,\quad   a_m \oplus {\bar u}_{\bar m} \longmapsto (a_m,  h(\bar m, \rho(a_m)) + {\bar u}_{\bar m}).
\end{equation*}
This means that, with the help of an Ehresmann connection $ h$, we can define also a Lie algebroid structure on $\widehat A$ with anchor map
\begin{equation*}
    \hat\rho = \bar\rho \circ \alpha_h^{-1}\colon \widehat A \longrightarrow \mathrm{T}\overline M,\quad a_m \oplus {\bar u}_{\bar m} \longmapsto   h(\bar m, \rho(a_m)) + {\bar u}_{\bar m}
\end{equation*}
and bracket
\begin{equation*}
    \llbracket{\hat\sigma}_1, {\hat\sigma}_2\rrbracket^{\widehat A} = \alpha_h \left( \llbracket  \alpha_h^{-1}({\hat\sigma}_1), \alpha_h^{-1}({\hat\sigma}_2)\rrbracket \right).
\end{equation*}

As before, sections of the type $\sigma\oplus 0$ (with $\sigma\in\Gamma(A)$) and $0\oplus V$ (with $V$ a vertical vector field on $\overline M$) are called projectable sections of $\widehat A$. 
In this notation
\begin{equation*}
    \alpha_h(\sigma,X)=\sigma \oplus \omega(X),\qquad \alpha_h^{-1}(\sigma \oplus V) = (\sigma,  h(\rho(\sigma))+V).
\end{equation*}
The anchor map $\hat\rho\colon \widehat A \to \mathrm{T}\overline M$  then has the property 
\begin{equation*}
    \hat\rho(\sigma\oplus 0) =  h(\rho(\sigma)), \qquad \hat\rho(0\oplus V) = V,
\end{equation*}

In what follows, we need the \emph{curvature} $K$ of the Ehresmann connection: For $X_i \in {\mathcal X}(M)$, it is defined by
\begin{equation*}
    K(X_1,X_2)\coloneqq [ h(X_1), h(X_2)]- h[X_1,X_2] =  \omega\left([ h(X_1), h(X_2)]\right).
\end{equation*}
With that, one may verify that the Lie algebroid brackets on projectable sections are, in fact, given by
\begin{equation}\label{eqn:extensionbracket1}
    \llbracket \sigma_1\oplus 0, \sigma_2\oplus 0\rrbracket^{\widehat A} = \llbracket\sigma_1,\sigma_2\rrbracket^A \oplus K(\rho(\sigma_1),\rho(\sigma_2)),\quad \llbracket \sigma_1\oplus 0, 0\oplus V_2\rrbracket^{\widehat A} = 0\oplus[ h(\rho(\sigma_1)),V_2],
\end{equation} 
\begin{equation}\label{eqn:extensionbracket2}
    \llbracket 0\oplus V_1, 0\oplus V_2\rrbracket^{\widehat A} = 0\oplus[V_1,V_2].
\end{equation}

We check that the above bracket relations are compatible with the anchor map. We consider 3 cases. First,
\begin{equation*}
    \hat\rho \llbracket \sigma_1\oplus 0, \sigma_2\oplus 0\rrbracket^{\widehat A} =  h(\rho\llbracket\sigma_1,\sigma_2\rrbracket^A)+ K(\rho(\sigma_1),\rho(\sigma_2)),\quad [\hat\rho(\sigma_1\oplus 0), \hat\rho(\sigma_2\oplus 0)] = [ h(\rho(\sigma_1)), h(\rho(\sigma_2))].
\end{equation*}
This is in agreement, since 
\begin{equation*}
    K(\rho(\sigma_1),\rho(\sigma_2)) = [ h(\rho(\sigma_1)), h(\rho(\sigma_2))]- h[\rho(\sigma_1),\rho(\sigma_2)] = [ h(\rho(\sigma_1)), h(\rho(\sigma_2))]- h(\rho\llbracket \sigma_1,\sigma_2\rrbracket^A),
\end{equation*}
because $\rho$ is a Lie algebra morphism.

Second,
\begin{equation*}
    \hat\rho \llbracket \sigma_1\oplus 0, 0\oplus V_2\rrbracket^{\widehat A} = [ h(\rho(\sigma_1)),V_2]=[\hat\rho(\sigma_1\oplus 0), \hat\rho(0\oplus V_2)],
\end{equation*}
and also the third holds easily. We conclude that we have $\hat{\rho}(\llbracket {\hat\sigma}_1, {\hat\sigma}_2\rrbracket^{\widehat A})=[\hat{\rho}({\hat\sigma}_1),\hat{\rho}({\hat\sigma}_2)]$ for ${\hat\sigma}_1,{\hat\sigma}_2\in \Gamma(\widehat{A})$, as required.

If we use coordinates $(x^i,y^\mu)$ on $\overline M$, then there exist functions $A^\mu_i(x,y)$ such that  
\begin{equation*}
    h_i\coloneqq h\left(\frac{\partial}{\partial x^i}\right) = \frac{\partial}{\partial x^i} + A^\mu_i \frac{\partial}{\partial y^\mu}.
\end{equation*}
With this, $[h_i,h_j] = K^\mu_{ij}\frac{\partial}{\partial y^\mu}$ and the components of the curvature are $K^\mu_{ij}=h_i(A^\mu_j)-h_j(A^\mu_i)$.

If we now introduce ${\hat e}_\alpha = e_\alpha\oplus 0$ and ${\widehat E}_\mu = 0\oplus \frac{\partial}{\partial y^\mu}$, then the anchor satisfies $\hat{\rho}(\hat{e}_{\alpha}) =\rho_{\alpha}^ih_i$ and $\hat{\rho}(\widehat{E}_{\mu})=\frac{\partial}{\partial y^{\mu}}$, and the expressions in \eqref{eqn:extensionbracket1} and \eqref{eqn:extensionbracket2} reduce to
\begin{equation*}\label{eqn:bracketshat}
    \llbracket \hat{e}_{\alpha},\hat{e}_{\beta}\rrbracket^{\widehat A}=L_{\alpha\beta}^\gamma {\hat e}_\gamma + \ M^\mu_{\alpha\beta}{\widehat E}_\mu,\qquad \llbracket \hat{e}_{\alpha},\widehat{E}_{\mu}\rrbracket^{\widehat A}= N^\nu_{\alpha\mu}{\widehat E}_\nu,\qquad \llbracket \widehat{E}_{\mu},\widehat{E}_{\nu}\rrbracket^{\widehat A}=0,
\end{equation*}
where $M^\mu_{\alpha\beta}\coloneqq \rho^i_\alpha\rho^j_\beta K^\mu_{ij}$ and $N^\nu_{\alpha\mu}\coloneqq - \rho^i_\alpha\frac{\partial A^\nu_i}{\partial y^\mu}$.


We end this section with a remark that is important for many of the applications that follow. When the fiber bundle $p\colon \overline M \to M$ is a trivial bundle, i.e., when $\overline M = M\times N$, then $\overline A = \rho^* \mathrm{T}\overline M \to \overline M$ can be identified with the bundle $A \times \mathrm{T}N \to M\times N$. On the other hand, in that case, we may always take $h$ to be a trivial connection, i.e., $A^\mu_i = 0$ and thus also $M^\mu_{\alpha\beta} = N^\mu_{\alpha\nu}=0$. This simplifies many of the above formulas. In particular, ${\widehat A}=p^{\ast}A\oplus_{\overline{M}} \mathrm{V}p \to \overline M$ becomes $p^{\ast}A\oplus_{\overline{M}}p^*\mathrm{T}N\simeq A\times \mathrm{T}N \to M \times N$ again.

\section{EDS and integral manifolds for Lie algebroids}\label{sec:EDSintman}
The goal of this section is to generalize the definitions of \cite{bryant1991exterior,ivey2016cartan} from the standard case to the case with Lie algebroids. From now on, we will make no notational distinction between the exterior derivative operator on $A$ (as defined in \Cref{sec:preliminaries}) and $\overline{A}$ (as given in \Cref{sec:pullback}), and we will simply denote the operator by $\delta$. In this section, we will exclusively deal with the exterior derivative on the algebroid $\overline{A}$.

Now, let $i\colon M\to \overline{M}$ be a section of $p$. Such sections are of the form $(x^i)\mapsto (x^i,\bar{y}^{\mu}(x))$ for some smooth $\bar{y}^{\mu}(x)$. With the help of $i$, we now wish to construct a section of $P$. 

\begin{definition}
For a section $i$ of $p$, we call the section $I$ of $P$, defined by
\begin{equation*}
    I(a_m)\coloneqq ( a_m, \mathrm{T}i(\rho(a_m)))
\end{equation*}
the \emph{$i$-induced section}. 
\end{definition}
The image of $I$ lies, indeed, in $\overline A= \rho^* \mathrm{T}{\overline M}$, because $i$ is a section of $p$ and
\begin{equation*}
    \mathrm{T}p( \mathrm{T}i(\rho(a_m))) = \mathrm{T}({\rm{id}}) (\rho(a_m)) = \rho(a_m).
\end{equation*}

\begin{proposition} \label{prop2}
The $i$-induced section $I$ is uniquely determined by the properties $P\circ I = {\rm id}$, $\bar{\rho}\circ I=\mathrm{T}i\circ \rho$ and $\bar{\tau}\circ I=i\circ \tau$.
\end{proposition}
\begin{proof} One easily verifies that, in local coordinates, $I$ is given by
\begin{equation*}
    (x^i,w^{\alpha})\longmapsto \left(x^i,y^{\mu}=\bar{y}^{\mu}(x),w^{\alpha},v^{\mu}=\rho_{\alpha}^i \frac{\partial \bar{y}^{\mu}}{\partial x^i} w^{\alpha}\right),
\end{equation*}
from which the statement easily follows.
\end{proof}



We have introduced $I$ here, because we can use it to define a pullback operator $I^{\ast}\colon \Lambda^q(\overline{A})\to \Lambda^q(A)$, as follows. On functions $\bar f\in C^{\infty}(\overline{M})$ and one-forms $\bar\theta$ on $\overline{A}$:
\begin{equation*}
    (I^{\ast}{\bar f})(m) = \bar f(i(m)),\qquad (I^{\ast}{\bar \theta})_m(a_m)= \bar\theta_{i(m)}(I(a_m)).
\end{equation*}
The local expressions are then
\begin{equation*}
    I^{\ast}x^i=x^i,\quad I^{\ast}y^{\mu}=\bar{y}^{\mu}(x),
\end{equation*}
\begin{equation*}
    I^{\ast}\bar{e}^{\alpha}=e^{\alpha}\in \Lambda^1(A),\quad I^{\ast} \overline{E}^{\mu}=\rho_{\alpha}^i\frac{\partial \bar{y}^{\mu}}{\partial x^i}e^{\alpha}\in \Lambda^1(A).
\end{equation*}
We can extend this to forms of arbitrary degree by imposing linearity and the rule $I^{\ast}(\omega\wedge \theta)=I^{\ast}\omega\wedge I^{\ast}\theta$ for all $\omega\in \Lambda^q(\overline{A})$ and $\theta\in \Lambda^r(\overline{A})$.


In the next proposition we use the definition of a Lie algebroid morphism between Lie algebroids over different base manifolds, as found in e.g., \cite{barbero2016invproblemliealgebroids,deleon2005survey}.

\begin{proposition}
The $i$-induced section $I\colon A \to \overline A$ is a Lie algebroid morphism, in the sense that $\delta\circ I^* = I^* \circ \delta$.
\end{proposition}
\begin{proof}


One may easily verify from the coordinate expressions that the property in the statement is satisfied, taking into account that, e.g., ${\delta} x^i=\rho_{\alpha}^i\bar{e}^{\alpha}$, $\delta y^\mu={\overline E}^\mu$,   ${\delta} \bar{e}^{\alpha}=-\frac{1}{2}L_{\beta \gamma}^{\alpha}\bar{e}^{\beta}\wedge \bar{e}^{\gamma}$ and ${\delta} \overline{E}^{\mu}=0$.
\end{proof}

\begin{definition}
A subspace $\mathcal{I}\subset \Lambda^{\ast}(\overline{A})$ is an \emph{algebraic ideal} if it is a direct sum of subspaces $\mathcal{I}^k\subset \Lambda^{\ast}(\overline{A})$ closed under the wedge product with arbitrary differential forms. In other words, it is a graded ideal of $\Lambda^{\ast}(\overline{A})$.
\end{definition}

\begin{definition}
A \emph{differential ideal} is an algebraic ideal $\mathcal{I}\subset \Lambda^{\ast}(\overline{A})$ such that $\delta \mathcal{I}\subset \mathcal{I}$.
\end{definition}

Given a set of forms of arbitrary degree $\{\omega^1,\dots,\omega^n\}\subset \Lambda^{\ast}(\overline{A})$, we denote by
\begin{equation*}
    \langle \omega^1,\dots,\omega^n\rangle_{\mathrm{alg}}\coloneqq \left\{\alpha^k\wedge \omega^k\mid \alpha^1,\dots,\alpha^n\in \Lambda^{\ast}(\overline{A})\right\}
\end{equation*}
the smallest algebraic ideal containing these forms, and by
\begin{equation*}
    \langle \omega^1,\dots,\omega^n\rangle\coloneqq \left\{\alpha^k\wedge \omega^k+\beta^k\wedge \delta \omega^k\mid \alpha^1,\dots,\alpha^n,\beta^1,\dots,\beta^n\in \Lambda^{\ast}(\overline{A})\right\}
\end{equation*}
the smallest differential ideal containing these forms. Note that if the algebraic ideal generated by $\{\omega^1,\dots,\omega^n\}\subset \Lambda^{\ast}(\overline{A})$ is in fact a differential ideal then $\langle \omega^1,\dots,\omega^n\rangle=\langle \omega^1,\dots,\omega^n\rangle_{\mathrm{alg}}$.

\begin{definition}
Let $i\colon M\to \overline{M}$ be a section of $p$. We say that $i$, or $i(M)\subset \overline{M}$ is an \emph{integral manifold} of a differential ideal $\mathcal{I}$ if for all $\theta\in \mathcal{I}\subset \Lambda^{\ast}(\overline{A})$, the equation $I^{\ast}\theta=0$ holds in $\Lambda^{\ast}(A)$.
\end{definition}

Our definition can be related to the standard definition of an integral manifold of a set of forms on a manifold, as it appears in e.g., \cite{ivey2016cartan}. 

\begin{definition} \cite{ivey2016cartan}
An \emph{exterior differential system with independence condition} on a manifold $\overline M$ consists of a differential ideal $\mathcal{I}\subset \Omega^{\ast}(\overline M)$ and a differential $n$-form $\Omega$ on $\overline{M}$. An \emph{integral manifold} of $(\mathcal{I},\Omega)$ is an $n$-dimensional submanifold $i\colon M\to \overline M$ such that $i^{\ast}\alpha=0$ for all $\alpha\in \mathcal{I}$, and $i^{\ast}\Omega\neq 0$ at each point of $M$. This last condition is called the \emph{independence condition.}
 \end{definition}

We can now relate it to our definition, in the following sense. If $p\colon \overline M \to M$ is a fiber bundle and $i\colon M \to \overline M$ is a section of $p$, then $i(M)$ is a submanifold of $\overline{M}$. We can take any volume form $\Omega_M$ on $M$ and set $\Omega\coloneqq p^{\ast} \Omega_M$. This is an $n$-form on $\overline M$ with the property that $i^{\ast} \Omega = (p\circ i)^{\ast} \Omega_M =\Omega_M \neq 0$. In this case, we take our Lie algebroids to be $A=\mathrm{T}M$ and then also $\overline{A}=\mathrm{T} \overline M$. In conclusion, the independence condition is implicit in our definition.  

The general definition of an integral manifold above does not make explicit use of a section of a fiber bundle. However, it depends on the choice of the form $\Omega$ in the following sense. If, for example on $\overline M = {\mathbb R}^3$, we consider $\Omega = \mathrm{d}x\wedge \mathrm{d}y$, then it follows from the independence condition that integral manifolds can be interpreted as maps $(x,y) \mapsto (x,y, {\bar z}(x,y))$, meaning that they are sections of the bundle projection $(x,y,z)\mapsto (x,y)$. But, if we take $\Omega = \mathrm{d}x\wedge \mathrm{d}z$, then integral manifolds become maps $(x,z) \mapsto (x,{\bar y}(x,z), z)$ and, therefore, now sections of the bundle projection $(x,y,z)\mapsto (x,z)$.

For this reason, we have opted to include the choice of a fiber bundle in our definition, over the choice of an independence condition.

The next example motivates and explains the importance of imposing that the ideal $\mathcal I$ is a differential ideal, when trying to find its integral manifolds.

\begin{example}\label{eg:oneforms}
Suppose that $\{\bar{e}^{\alpha},\overline{E}^{\mu}\}$ is the dual basis of $\{\bar{e}_{\alpha},\overline{E}_{\mu}\}$. Let $\mathcal{I}=\langle \sigma^1,\dots,\sigma^n\rangle$, where $\sigma^k\in \Lambda^1(\overline{A})$ is of the form $\sigma^k=\theta_{\alpha}^k \bar{e}^{\alpha}+\varpi_{\mu}^k \overline{E}^{\mu}$. For a section $i\colon M\to \overline M, (x^i) \mapsto (x^i, \bar{y}^\mu(x))$, one finds that
\begin{equation*}
    I^{\ast}\sigma^k=\theta_{\alpha}^k e^{\alpha}+\rho_{\alpha}^i\frac{\partial \bar{y}^{\mu}}{\partial x^i}\varpi_{\mu}^k e^{\alpha}.
\end{equation*}
Therefore, to find integral manifolds $i$ of $\mathcal{I}$, we need to solve the system of PDEs
\begin{equation}\label{eqn:PDEintmanifolds}
    \theta_{\alpha}^k(x,\bar{y})+\rho_{\alpha}^i \frac{\partial \bar{y}^{\mu}}{\partial x^i}\varpi_{\mu}^k(x,\bar{y})=0,
\end{equation}
for the functions $\bar{y}^{\mu}(x)$. 

Recall that the anchor map $\rho\colon A \to M$ defines a Lie algebra morphism. From this, we can see that the equations \eqref{eqn:PDEintmanifolds} are not independent from one another and we may therefore derive an obstruction to the existence of a solution ${\bar y}^\mu(x)$. Indeed, if we set, just for brevity, $X_\alpha = \rho^i_\alpha \frac{\partial}{\partial x^i}$ and $Y_\alpha=X_\alpha+ X_\alpha({\bar y}^\nu)\frac{\partial }{\partial y^\nu}$, then any solution of \Cref{eqn:PDEintmanifolds} also satisfies
\begin{equation*}
    X_\alpha (X_\beta (\bar{y}^\mu))\varpi^k_\mu = - {Y}_\alpha(\theta^k_\beta) - X_\beta(\bar{y}^\mu){Y}_\alpha(\varpi^k_\mu),
\end{equation*}
and likewise when $\alpha$ and $\beta$ are interchanged.  On the other hand, we also have
\begin{equation*}
    L_{\alpha\beta}^\gamma X_\gamma({\bar y}^\mu)\varpi^k_\mu = - L_{\alpha\beta}^\gamma \theta^k_\gamma.
\end{equation*}
In view of the identities \eqref{eqn:anchorspecialidentity}, or in the current notation, $[X_\alpha,X_\beta] = L_{\alpha\beta}^\gamma X_\gamma$, we need that a solution ${\bar y}^\mu(x)$ of \Cref{eqn:PDEintmanifolds} also should satisfy
\begin{equation} \label{eq:integ}
0= - Y_\alpha(\theta^k_\beta) - X_\beta(\bar{y}^\mu)Y_\alpha(\varpi^k_\mu) +  Y_\beta(\theta^k_\alpha) + X_\alpha(\bar{y}^\mu)Y_\beta(\varpi^k_\mu) +  L_{\alpha\beta}^\gamma \theta^k_\gamma. 
\end{equation}
This condition on a solution ${\bar y}^\mu(x)$ is not automatically satisfied for any given set of forms $\sigma^k$, and we call it the \emph{dependency condition}. 
When it is assumed that $\mathcal{I}$ forms a differential ideal, we can show that it is satisfied for any solution of \Cref{eqn:PDEintmanifolds}.

\begin{proposition} \label{prop:depcond}
When $\mathcal{I}=\langle \sigma^1,\dots,\sigma^n\rangle_{\mathrm{alg}}$ forms a differential ideal, the dependency condition \eqref{eq:integ} of \Cref{eqn:PDEintmanifolds} holds identically.
\end{proposition}
\begin{proof}
Suppose that $\langle \sigma^1,\dots,\sigma^n\rangle_{\mathrm{alg}}$ is a differential ideal. Then there exist one-forms $\nu_l^k=\nu_{l\alpha}^k \bar{e}^{\alpha}+\nu_{l\mu}^k \overline{E}^{\mu}\in \Lambda^1(\overline{A})$ such that $\delta \sigma^k=\nu_l^k\wedge \sigma^l$. Expanding this expression in terms of the basis $\{\bar{e}^{\alpha},\overline{E}^{\mu}\}$ gives
\begin{equation*}
    \delta \sigma^k=\nu_{l\alpha}^k \theta_{\beta}^l \bar{e}^{\alpha}\wedge \bar{e}^{\beta}+(\nu_{l\alpha}^k \varpi_{\mu}^l-\nu_{l\mu}^k \theta_{\alpha}^l) \bar{e}^{\alpha}\wedge \overline{E}^{\mu}+\nu_{l\mu}^k \varpi_{\nu}^l \overline{E}^{\mu}\wedge \overline{E}^{\nu}.
\end{equation*}
On the other hand, computing $\delta \sigma^k$ directly gives
\begin{align*}
    \delta \sigma^k&=\delta(\theta_{\alpha}^k \bar{e}^{\alpha}+\varpi_{\mu}^k \overline{E}^{\mu})\\&=\delta \theta_{\alpha}^k\wedge \bar{e}^{\alpha}+\theta_{\alpha}^k \delta \bar{e}^{\alpha}+\delta \varpi_{\mu}^k\wedge \overline{E}^{\mu}+\varpi_{\mu}^k \delta \overline{E}^{\mu}\\&=\left(\rho_{\beta}^i \frac{\partial \theta_{\alpha}^k}{\partial x^i} \bar{e}^{\beta}+\frac{\partial \theta_{\alpha}^k}{\partial y^{\mu}}\overline{E}^{\mu}\right)\wedge \bar{e}^{\alpha}-\frac{1}{2} \theta_{\alpha}^k L_{\beta \gamma}^{\alpha}\bar{e}^{\beta}\wedge \bar{e}^{\gamma}+\left(\rho_{\alpha}^i \frac{\partial \varpi_{\mu}^k}{\partial x^i} \bar{e}^{\alpha}+\frac{\partial \varpi_{\mu}^k}{\partial y^{\nu}}\overline{E}^{\nu}\right)\wedge \overline{E}^{\mu}\\&=\left(\rho_{\alpha}^i \frac{\partial \theta_{\beta}^k}{\partial x^i}-\frac{1}{2} \theta_{\gamma}^k L_{\alpha \beta}^{\gamma}\right)\bar{e}^{\alpha}\wedge \bar{e}^{\beta}+\left(\rho_{\alpha}^i \frac{\partial \varpi_{\mu}^k}{\partial x^i}-\frac{\partial \theta_{\alpha}^k}{\partial y^{\mu}}\right)\bar{e}^{\alpha}\wedge \overline{E}^{\mu}+\frac{\partial \varpi_{\nu}^k}{\partial y^{\mu}}\overline{E}^{\mu}\wedge \overline{E}^{\nu}.
\end{align*}
It follows that
\begin{equation}\label{eqn:integrability1}
    \nu_{l\mu}^k \varpi_{\nu}^l-\nu_{l\nu}^k \varpi_{\mu}^l=\frac{\partial \varpi_{\nu}^k}{\partial y^{\mu}}-\frac{\partial \varpi_{\mu}^k}{\partial y^{\nu}},
\end{equation}
\begin{equation}\label{eqn:integrability2}
    \nu_{l\alpha}^k \varpi_{\mu}^l-\nu_{l\mu}^k \theta_{\alpha}^l= X_\alpha(\varpi_{\mu}^k)-\frac{\partial \theta_{\alpha}^k}{\partial y^{\mu}},
\end{equation}
\begin{equation}\label{eqn:integrability3}
    X_\alpha(\theta_{\beta}^k)-X_\beta(\theta_{\alpha}^k)-\theta_{\gamma}^k L_{\alpha \beta}^{\gamma}=\nu_{l\alpha}^k \theta_{\beta}^l-\nu_{l\beta}^k\theta_{\alpha}^l.
\end{equation}
We now compute the right hand side of \eqref{eq:integ} and show that it vanishes. In view of \eqref{eqn:integrability2} and \eqref{eqn:integrability3}, we get
\begin{eqnarray*}
\mbox{RHS} &=& -\nu_{l\alpha}^k \theta_{\beta}^l+\nu_{l\beta}^k\theta_{\alpha}^l 
- X_\alpha({\bar y}^\mu)\frac{\partial\theta^k_\beta}{\partial y^\mu}
+  X_\beta({\bar y}^\mu)\frac{\partial\theta^k_\alpha}{\partial y^\mu}
\\&&- X_\beta({\bar y}^\mu) \left( \nu_{l\alpha}^k \varpi_{\mu}^l-\nu_{l\mu}^k \theta_{\alpha}^l   + \frac{\partial \theta_{\alpha}^k}{\partial y^{\mu}} \right)
+ X_\alpha({\bar y}^\mu) \left( \nu_{l\beta}^k \varpi_{\mu}^l-\nu_{l\mu}^k \theta_{\beta}^l   + \frac{\partial \theta_{\beta}^k}{\partial y^{\mu}} \right)
\\&& - X_\beta(\bar{y}^\mu)X_\alpha({\bar y}^\nu)\frac{\partial\varpi^k_\mu}{\partial y^\nu} +  X_\alpha(\bar{y}^\mu)X_\beta({\bar y}^\nu)\frac{\partial\varpi^k_\mu}{\partial y^\nu}.
\end{eqnarray*}
After simplification, and after interchanging $\nu$ and $\mu$ in the last term and using \eqref{eqn:integrability1} for the last two terms, we obtain
\begin{eqnarray*}
\mbox{RHS} &=& -\nu_{l\alpha}^k \theta_{\beta}^l+\nu_{l\beta}^k\theta_{\alpha}^l 
- X_\beta({\bar y}^\mu) \left( \nu_{l\alpha}^k \varpi_{\mu}^l-\nu_{l\mu}^k \theta_{\alpha}^l    \right)
+ X_\alpha({\bar y}^\mu) \left( \nu_{l\beta}^k \varpi_{\mu}^l-\nu_{l\mu}^k \theta_{\beta}^l    \right)
\\&& + X_\beta(\bar{y}^\mu)X_\alpha({\bar y}^\nu)\left(\nu_{l\mu}^k \varpi_{\nu}^l-\nu_{l\nu}^k \varpi_{\mu}^l\right).
\end{eqnarray*}
If we now take into account that ${\bar y}^\mu(x)$ is a solution of \Cref{eqn:PDEintmanifolds}, we easily see that 
$\mbox{RHS} =0$, as claimed.
\end{proof}

\end{example}

\section{An application to a single semilinear PDE}\label{sec:applications}

One example of the system in \eqref{eqn:PDEintmanifolds} is that of a first-order semilinear partial differential equation of the form
\begin{equation}\label{eqn:characteristics}
    a(x,y)u_x+b(x,y)u_y=c(x,y,u).
\end{equation}
This type of PDE can typically be solved using the method of characteristics, and it is also discussed in e.g., \cite{ivey2016cartan}. In this section, we wish to take an alternative approach. We will show that finding solutions to this PDE is equivalent to finding integral manifolds of a certain differential ideal on a very specific Lie algebroid.

We follow the construction in \Cref{sec:pullback}. We take $M=\mathbb{R}^2$ with coordinates $(x,y)$ and $\overline{M}=\mathbb{R}^3$ with coordinates $(x,y,u)$. For the vector bundle of our Lie algebroid, we take $A=\mathbb{R}^3=\mathbb{R}^2\times \mathbb{R}$ with coordinates $(x,y,w)$. Then, $\overline{A}=\mathbb{R}^5=\mathbb{R}^3\times \mathbb{R}^2=\{(x,y,u,w,{\mathcal U})\}$ and the following diagram commutes:
\begin{equation*}
\begin{CD}
\overline{A}=\mathbb{R}^5@>P>>A=\mathbb{R}^3\\
@V\bar{\tau} VV @VV\tau V\\
\overline{M}=\mathbb{R}^3@>>p>M=\mathbb{R}^2.\\
\end{CD}
\end{equation*}

This choice of fiber bundle corresponds to the choice of independence condition $i^\ast(\mathrm{d}x\wedge \mathrm{d}y)\neq 0$, which is commonly adopted for examples of this kind (see \cite{ivey2016cartan}).

We shall now describe the Lie algebroid structure of $A$ and $\overline{A}$. We define the following anchor maps
\begin{equation*}
    \rho\colon \mathbb{R}^3\longrightarrow \mathbb{R}^4,\quad (x,y,w)\longmapsto (x,y,\dot{x}=a(x,y)w,\dot{y}=b(x,y)w),
\end{equation*}
\begin{equation*}
    \bar{\rho}\colon \mathbb{R}^5\longrightarrow \mathbb{R}^6,\quad (x,y,u,w,\mathcal{U})\longmapsto (x,y,u,\dot{x}=a(x,y)w,\dot{y}=b(x,y)w,\dot{u}=\mathcal{U}),
\end{equation*}
for $A$ and $\overline{A}$, respectively. Now, let $\{e_w\}$ be the section of $A$ given by
\begin{equation*}
    e_w\colon \mathbb{R}^2\longrightarrow \mathbb{R}^3,\quad (x,y)\longmapsto (x,y,1).
\end{equation*}
This is a globally defined section and $\{e_w\}$ is a basis of sections of $A$. Let $\{e^w\}$ be the dual basis of $1$-forms on $A$. We can also give a basis of sections of $\overline{A}$ as follows. Let
\begin{equation*}
    \bar{e}_w\colon (x,y,u)\longmapsto (x,y,u,1,0),
\end{equation*}
\begin{equation*}
    \bar{e}_{\mathcal{U}}\colon (x,y,u)\longmapsto (x,y,u,0,1).
\end{equation*}
Similarly, let $\{\bar{e}^w,\bar{e}^{\mathcal{U}}\}$ be the dual basis of $1$-forms on $\overline{A}$. On $A$, we define a bracket $\llbracket \cdot,\cdot \rrbracket^A$ by
\begin{equation*}
    \llbracket e_w,e_w \rrbracket^A=0.
\end{equation*}
It is compatible with the anchor map because $[\rho(e_w),\rho(e_w)]=0=\rho(\llbracket e_w,e_w\rrbracket^A)$. On $\overline{A}$, we hence have
\begin{equation*}
    \llbracket \bar{e}_w,\bar{e}_w\rrbracket=0,\quad \llbracket \bar{e}_w,\bar{e}_{\mathcal{U}}\rrbracket=0,\quad \llbracket \bar{e}_{\mathcal{U}},\bar{e}_{\mathcal{U}}\rrbracket=0.
\end{equation*}

On the Lie algebroid $\overline A$, we consider the ideal $\mathcal{I}=\langle \theta\rangle$, where $\theta\coloneqq \bar{e}^{\mathcal{U}}-c\bar{e}^w$. Integral manifolds of $\mathcal{I}$ are represented by sections $i\colon (x,y) \mapsto (x,y,u(x,y))$ whose induced section is
\begin{equation*}
    I\colon \mathbb{R}^3\longrightarrow \mathbb{R}^5,\quad (x,y,w)\longmapsto \left(x,y,u(x,y),w,\frac{\partial u}{\partial x}(x,y)\cdot a(x,y)w+\frac{\partial u}{\partial y}(x,y)\cdot b(x,y)w\right).
\end{equation*}
Since we are in a situation similar to \Cref{eg:oneforms}, the PDE whose solutions $u(x,y)$ describe integral manifolds is \Cref{eqn:PDEintmanifolds}, which in this case becomes \Cref{eqn:characteristics}.
  
We now verify that the algebraic ideal generated by $\theta$ is in fact a differential ideal. Denote by $\delta$ the exterior derivative operator for the Lie algebroid $\overline{A}$. Using \Cref{eqn:structure}, one sees that
\begin{equation*}
    \delta \bar{e}^w=0,\quad \delta \bar{e}^{\mathcal{U}}=0.
\end{equation*}
From this,
\begin{equation*}
    \delta \theta=-\delta c\wedge \bar{e}^w =-c_u \bar{e}^{\mathcal{U}}\wedge \bar{e}^w =c_u \bar{e}^w\wedge \bar{e}^{\mathcal{U}}=c_u \bar{e}^w\wedge \theta,
\end{equation*}
where we denote by $c_x\coloneqq \frac{\partial c}{\partial x}$, $c_y\coloneqq \frac{\partial c}{\partial y}$, $c_u\coloneqq \frac{\partial c}{\partial u}$. Hence $\mathcal{I}=\langle \theta\rangle_{\mathrm{alg}}$.

\section{An application to dynamical systems with symmetry}\label{sec:symmetry}
In this section we discuss the Atiyah Lie algebroid and its prolongation over a fiber bundle. We then use this to study dynamical systems with a symmetry group using EDS on Lie algebroids.


\subsection{The Atiyah Lie algebroid} \label{sec:atiyahalgebroid}

Let $\Phi\colon G\times Q\to Q$ be a free and proper action of a Lie group $G$ on a manifold $Q$. Then, $\pi\colon Q\to Q\big\slash G$ defines a principal fiber bundle. By setting $\Phi_g\colon Q\to Q, q\mapsto \Phi(g,q)$, we can extend the action on $Q$ to a (free and proper) tangent-lifted action 
\begin{equation*}
    G\times \mathrm{T}Q\longrightarrow \mathrm{T}Q,\quad (g,v_q)\longmapsto \mathrm{T}\Phi_g(v_{q}),
\end{equation*}
with corresponding induced projection $\pi^{\mathrm{T}Q}\colon \mathrm{T}Q\to \mathrm{T}Q\big\slash G$.

It is well known that the vector bundle $\mathrm{T}Q \big\slash G\to Q\big\slash G$ can be given the structure of a Lie algebroid, known as the \emph{Atiyah Lie algebroid} (see e.g., \cite{deleon2005survey,mackenzie1987algebroid} for more details). It is easy to see that the sections $\widehat{X}\colon Q\big\slash G \to \mathrm{T}Q \big\slash G$ of the bundle are in one-to-one correspondence with $G$-invariant vector fields $X\colon Q \to TQ$, by means of the relation $\widehat{X} \circ \pi = \pi^{\mathrm{T}Q}\circ X$. Since the Lie bracket of two invariant vector fields is again an invariant vector field, this defines a Lie bracket of sections of $\mathrm{T}Q \big\slash G\to Q\big\slash G$. This Lie bracket is compatible with the anchor map  $\varrho \colon \mathrm{T}Q\big\slash G\to \mathrm{T}(Q\big\slash G), \pi^{\mathrm{T}Q}(v_{q})\mapsto \mathrm{T}\pi(v_q)$.

If $f$ is a $G$-invariant function on $Q$, then we write $\hat f$ for its corresponding function on $Q\big\slash G$, i.e., $\hat f\circ\pi = f$. The one-to-one correspondence between sections $\widehat{X}$ of the Atiyah Lie algebroid and invariant vector fields $X$ leads to a similar statement in the context of forms. If $\theta$ is a $G$-invariant form on $Q$, then $\theta(X)$ is a $G$-invariant function on $Q$ which may therefore be identified with a function on $Q\big\slash G$. We can use this property to define a section $\hat\theta$ of $(\mathrm{T}Q\big\slash G)^{\ast} \to Q\big\slash G$ (the dual of the Atiyah Lie algebroid), by
\begin{equation*}
    \hat\theta(\widehat{X})\coloneqq \widehat{\theta(X)}.
\end{equation*}
If we denote by $\delta$ the exterior derivative of the Atiyah Lie algebroid, and by $\rm{d}$ for the one on $Q$, then the construction of the anchor map and the Lie algebroid bracket are such that
\begin{equation*}
    \delta\hat\theta = \widehat{\rm{d}\theta}.
\end{equation*}

Now, let $k\colon \overline{Q}\to Q$, $\bar{q}\mapsto q$ be a fiber bundle, whose total space also comes with a free and proper $G$-action $\overline{\Phi}\colon G\times \overline{Q}\to \overline{Q}$. We will assume that it projects on the underlying action $\Phi\colon G\times Q\to Q$ (so that $k\circ \overline{\Phi}_g=\Phi_g\circ k$), with notation interpreted in the usual way.

In this setup, we have two Atiyah Lie algebroids, namely $\mathrm{T}\overline Q\big\slash G$ and $\mathrm{T}Q\big\slash G$.  Let $\bar{\pi}\colon \overline{Q}\to \overline{Q}\big\slash G$ be the orbit map. Define a map
\begin{equation*}
    p\colon \overline{M} = \overline{Q}\big\slash G\longrightarrow M=Q\big\slash G,\quad \bar{\pi}(\bar{q})\longmapsto \pi(k(\bar{q})).
\end{equation*}
It is well-defined, because
\begin{equation*}
    \pi(k(g\cdot \bar{q}))=\pi(g\cdot k(\bar{q}))=\pi(k(\bar{q})).
\end{equation*}
In this subsection, we wish to show that if we take $A=\mathrm{T}Q\big\slash G$ in \eqref{eqn:cd}, then the $p$-prolongation of the Atiyah Lie algebroid $A=\mathrm{T}Q\big\slash G$ (i.e., $\overline{A}=\varrho^{\ast}(\mathrm{T}(\overline Q \big\slash G))$) can be identified with the Atiyah Lie algebroid $\mathrm{T}\overline Q\big\slash G \to \overline Q\big\slash G$.

The anchor map of the Atiyah algebroid of $\overline Q$ is $\bar\varrho\colon \mathrm{T}\overline{Q}\big\slash G\to \mathrm{T}(\overline{Q}\big\slash G)$, $\bar{\pi}^{\mathrm{T}\overline{Q}}(\bar{v}_{\bar{q}})\mapsto \mathrm{T}\bar{\pi}(\bar{v}_{\bar{q}})$, where $\bar{\pi}^{\mathrm{T}\overline{Q}}\colon \mathrm{T}\overline{Q}\to \mathrm{T}\overline{Q}\big\slash G$ is the canonical projection. Since $k\circ \overline{\Phi}_g=\Phi_g\circ k$, also $\mathrm{T}k\circ \mathrm{T}\overline{\Phi}_g =\mathrm{T}{\Phi}_g\circ \mathrm{T}k$. From this, we may define a map $[\mathrm{T}k]\colon\mathrm{T}\overline{Q}\big\slash G \to \mathrm{T}{Q}\big\slash G$. Together, we get a map
\begin{equation}\label{eq:alphaiso}
    \alpha\colon \mathrm{T}\overline{Q}\big\slash G \longrightarrow \overline A,\quad [{\bar v}_{\bar q}] \longmapsto ( a_m \coloneqq [\mathrm{T}k]([{\bar v}_{\bar q}]) , {\bar v}_{\bar m} \coloneqq \bar\varrho([{\bar v}_{\bar q}]) = \mathrm{T}{\bar\pi}({\bar v}_{\bar q})).
\end{equation}
We claim that $\alpha$ is well-defined. To this end, we verify that $\varrho(a_m)=\mathrm{T}p({\bar v}_{\bar m})$. Since $p\circ \bar{\pi}=\pi\circ k$, the projections $\bar{\pi}^{\mathrm{T}\overline{Q}}\colon \mathrm{T}\overline{Q}\to \mathrm{T}\overline{Q}\big\slash G$ and $\pi^{\mathrm{T}Q}\colon \mathrm{T}Q\to \mathrm{T}Q\big\slash G$ make the following diagrams commute:
\begin{equation*}\label{eqn:cd2}
\begin{CD}
\mathrm{T}\overline{Q}@>\mathrm{T}k>>\mathrm{T}Q\\
@V{\mathrm{T}\bar{\pi}}VV @VV{\mathrm{T}\pi} V\\
\mathrm{T}(\overline{Q}\big\slash G)@>>\mathrm{T}p>\mathrm{T}(Q\big\slash G).\\
\end{CD}
\end{equation*}
It then follows that, indeed,
\begin{equation*}
    \mathrm{T}p({\bar v}_{\bar m}) = \mathrm{T}p(   \mathrm{T}{\bar\pi}({\bar v}_{\bar q})) = \mathrm{T}\pi(   \mathrm{T}{k}({\bar v}_{\bar q})) =\varrho([\mathrm{T}k]([{\bar v}_{\bar q}])) = \varrho(a_m).
\end{equation*}

We now use Ehresmann connections to construct an inverse of $\alpha$. Assume that $h_p\colon p^{\ast} \mathrm{T}(Q\big\slash G) \to \mathrm{T}(\overline Q\big\slash G )$ is a connection on $p$ and let $\omega_p\colon \mathrm{T}{\overline Q}\big\slash G \to \mathrm{V}p$, $u_{[\bar q]} \mapsto u_{[\bar q]} - h_p([\bar q], \mathrm{T}p(u_{[\bar q]}))$ be the left splitting of $h_p$. Then,   we can also extend $h_p$ to an (equivariant) Ehresmann connection $h_k\colon k^{\ast} (\mathrm{T}{Q}) \to \mathrm{T}\overline{Q}$ on $k$ by the rule $\mathrm{T}\bar\pi\circ h_k = h_p \circ \mathrm{T}\pi$ (written in a somewhat shortened version). Likewise, if $h_\pi\colon \pi^{\ast} \mathrm{T}(Q\big\slash G) \to \mathrm{T}Q$ is a principal connection on $\pi$, then we can extend it to a principal connection $h_{\bar\pi}\colon {\bar\pi}^{\ast} \mathrm{T}(\overline Q\big\slash G) \to \mathrm{T}\overline Q$, by using $\mathrm{T}k\circ h_{\bar\pi} = h_\pi \circ \mathrm{T}p$.
If we set
\begin{equation*}
    \beta\colon  \overline A \longrightarrow \mathrm{T}\overline{Q}\big\slash G,\quad ( [v_q] , u_{[\bar q]}) \longmapsto [{\bar v}_{\bar q}]
\end{equation*}
with
\begin{equation*}
    {\bar v}_{\bar q}\coloneqq h_k(\bar q, v_q  ) + h_{\bar\pi}(\bar q,\omega_p(u_{[\bar q]})),
\end{equation*}
then this map is independent of the choice of representative $\bar{q}$ of $[\bar q]$, due to equivariance of $h_k$ and $h_{\bar{\pi}}$.

The map $\beta$ is the inverse of $\alpha$ because, by definition of the prolongation algebroid, $\mathrm{T}p (u_{[\bar q]}) = \mathrm{T}\pi(v_q)$, and
\begin{eqnarray*}
\bar\varrho([{\bar v}_{\bar q}]) &=& \mathrm{T}\bar\pi({\bar v}_{\bar q}) =   (\mathrm{T}\bar\pi\circ h_k)(\bar q, v_q) + \omega_p(u_{[\bar q]})  = h_p([\bar q], \mathrm{T}\pi(v_q)) +   \omega_p(u_{[\bar q]})\\& =&  h_p([\bar q], \mathrm{T}p(u_{[\bar q]})) +   \omega_p(u_{[\bar q]}) = u_{[\bar q]}.
\end{eqnarray*}
Next, because $\mathrm{T}p(\omega_p(u_{[\bar q]}))=0$,
\begin{equation*}
    \mathrm{T}k({\bar v}_{\bar q}) = v_q + (\mathrm{T}k \circ h_{\bar\pi})(\bar q,\omega_p(u_{[\bar q]})) = v_q +  h_{\pi}(  q,  \mathrm{T}p(\omega_p(u_{[\bar q]}))) =  v_q,
\end{equation*}
and therefore $[\mathrm{T}k]([{\bar v}_{\bar q}]) =  [v_q]$.

\begin{proposition}\label{prop4}
The $p$-prolongation $\overline A$ of the Atiyah Lie algebroid $A=\mathrm{T}Q\big\slash G$ is isomorphic to the Atiyah Lie algebroid $\mathrm{T}\overline Q \big\slash G$. 
\end{proposition}
\begin{proof} From the above, we already know that the vector bundles are isomorphic and that their anchor maps agree. We now show that the map $\alpha$ in \eqref{eq:alphaiso} is a Lie algebroid isomorphism.

Since the bracket of the prolongation bundle $\overline A$ is defined component-wise for projectable sections in \eqref{eq:defprolbracket}, and both projections $P$ and $\bar{\rho}$ act as Lie algebra morphisms to the respective components $\mathrm{T} Q\big\slash G$ and $\mathrm{T}(\overline Q\big\slash G)$, we only need to check that the bracket of the Atiyah algebroid $\mathrm{T}\overline Q\big\slash G$ has the same property, when we consider its projections on these components.

First, consider the anchor $\bar{\varrho}\colon \mathrm{T}\overline Q\big\slash G \to \mathrm{T}(\overline Q\big\slash G)$. Since $\mathrm{T}\overline Q\big\slash G$ is a Lie algebroid, its anchor map is by definition a Lie algebra morphism. 

Second, consider $[\mathrm{T}k]\colon \mathrm{T}\overline Q\big\slash G \to \mathrm{T} Q\big\slash G$. Note that the bracket of two $k$-related vector fields is again $k$-related. For this reason, $\mathrm{T}k\colon \mathrm{T}\overline Q \to \mathrm{T} Q$ is always a Lie algebra morphism. Moreover, $k$ maps invariant vector fields on $\overline Q$ onto invariant vector fields on $Q$. Since the Lie algebroid brackets on $\mathrm{T}\overline Q\big\slash G$ and $\mathrm{T} Q\big\slash G$ are both defined by the restriction of the Lie bracket of vector fields to invariant vector fields, $[\mathrm{T}k]$ is also a Lie algebra morphism.
\end{proof}


Hence, the diagram \eqref{eqn:cd} takes the form
\begin{equation*}
\begin{CD}
\mathrm{T}\overline{Q}\big\slash G@>P>>\mathrm{T}Q\big\slash G\\
@V{\bar{\tau}}VV @VV\tau V\\
\overline{Q}\big\slash G@>>p>Q\big\slash G.\\
\end{CD}
\end{equation*}

Suppose now that we have a section $i\colon Q\big\slash G \to {\overline Q}\big\slash G$ and a $G$-equivariant map $\hat i\colon Q \to \overline Q$, with the property $\bar\pi\circ\hat i = i \circ \pi$, i.e., the diagram
\begin{equation*}
\begin{CD}
Q@>\hat i>>\overline Q\\
@V\pi VV @VV \bar\pi V\\
Q\big\slash G@>>i>{\overline Q}\big\slash G,\\
\end{CD}
\end{equation*}
commutes.

\begin{proposition} \label{prop:inducedsectionAtiyah} The $i$-induced section $I$ on the prolongation $\overline A$ of $A$ by $p$ is  $[\mathrm{T}\hat{i}]$.
\end{proposition}
\begin{proof} 
From the previous proposition, we know that the $p$-prolongation $\overline A$ of $A$ is the Atiyah Lie algebroid $\mathrm{T}\overline Q\big\slash G$.
Following \Cref{prop2}, it is enough to show that  $\bar{\varrho}\circ [\mathrm{T}\hat{i}]=\mathrm{T}i\circ \varrho$ and $\bar{\tau}\circ [\mathrm{T}\hat{i}]=i\circ \tau$. Note that if $v_q\in \mathrm{T}Q$ and $[v_q]\in \mathrm{T}Q\big\slash G$ is its equivalence class, we have that
\begin{equation*}
    \bar{\varrho}\circ [\mathrm{T}\hat{i}]([v_q])=\bar{\varrho}([\mathrm{T}\hat{i}(v_q)])=\mathrm{T}\bar{\pi}(\mathrm{T}\hat{i}(v_q))=\mathrm{T}(\bar{\pi}\circ \hat{i})(v_q)=\mathrm{T}(i\circ \pi)(v_q)=\mathrm{T}i\circ \mathrm{T}\pi(v_q)=\mathrm{T}i\circ \varrho([v_q]).
\end{equation*}
Let $\mu_{\overline Q}\colon \mathrm{T}\overline Q\to \overline Q$, $\bar{v}_{\bar{q}}\mapsto \bar{q}$ and $\mu_Q\colon \mathrm{T}Q\to Q$, $v_q\mapsto q$ be the tangent bundles of ${\overline Q}$ and $Q$, respectively. Note that the projections of the Atiyah Lie algebroids $\bar{\tau}\colon \mathrm{T}\overline Q\big\slash G\to \overline Q\big\slash G$ and $\tau\colon \mathrm{T}Q\big\slash G\to {Q\big\slash G}$ are given by $[\bar{v}_{\bar q}]\mapsto [\bar q]$ and $[v_q]\mapsto [q]$, respectively. Hence
\begin{equation*}
    \bar{\tau}([\mathrm{T}\hat{i}]([v_q]))=\bar{\tau}([\mathrm{T}\hat{i}(v_q)])=[\mu_{\overline Q}((\mathrm{T}\hat{i})(v_q))]=[\hat{i}(q)]=i([q]).
\end{equation*}
On the other hand,
\begin{equation*}
    i\circ \tau([v_q])=i([\mu_{Q}(v_q)])=i([q]).\qedhere
\end{equation*}
\end{proof}

\subsection{Systems with symmetry and reconstruction of integral manifolds}

Let $P$ be a manifold and let $\Psi\colon G\times P\to P$ be a $G$-action on $P$. Then $\pi\colon P\to P\big\slash G$ is a principal fiber bundle. Fix a principal connection whose left splitting is given by $\omega\colon \mathrm{T}P\to \mathfrak{g}$ (here $\mathfrak{g}$ is the Lie algebra of $G$). Let $Y\in \mathcal{X}(P)$ be a $G$-invariant vector field. Then there exists a unique \emph{reduced vector field} $\overline{Y}\in \mathcal{X}(P\big\slash G)$ such that
\begin{equation*}\label{reducedvf}
    \overline{Y}(\pi(x))=\mathrm{T}\pi(Y(x)),
\end{equation*}
for all $x\in P$. Now fix a point $x_0\in P$, and let $[x_0]\in P\big\slash G$ be its corresponding equivalence class. Let $\bar{\gamma}_{[x_0]}$ be an integral curve of $\overline{Y}$ with $\bar{\gamma}(0)=[x_0]$. Then $\gamma_{x_0}$ defined by
\begin{equation*}
    \gamma_{x_0}(t)\coloneqq \Psi_{g(t)}\bar{\gamma}_{x_0}^H(t)
\end{equation*}
is the unique integral curve of $Y$ with $\gamma_{x_0}(0)=x_0$, where $\bar{\gamma}_{x_0}^H$ is the horizontal lift of $\bar{\gamma}_{[x_0]}$  to $x_0\in P$, and $g(t)$ solves the reconstruction equation
\begin{equation*}
    \theta(\dot{g})=\omega(Y\circ  \bar{\gamma}_{x_0}^H),
\end{equation*}
where $\theta$ is the Maurer--Cartan form of $G$, with the initial condition $g(0)=e$ ($e$ is the identity element of $G$) \cite{firstorderrec,ellis2011reconstruction}.

In comparison with the previous subsection, we take $Q=\mathbb{R}\times G$ and $\overline Q={\mathbb R}\times P$. The fiber bundle $k$ of \Cref{sec:atiyahalgebroid} is then $k\colon \mathbb{R}\times P\to \mathbb{R}\times G$. Moreover, we can make the identification $M=Q\big\slash G=(\mathbb{R}\times G)\big\slash G\simeq \mathbb{R}$ and $\overline M=\overline Q \big\slash G={\mathbb R}\times P \big\slash G$.

When we suppose that $P$ trivializes, i.e., when there exists an open set $U\subset P\big\slash G$ so that $\pi^{-1}(U)\simeq U\times G$, then we may write $x=([x],g)$ for some $g\in G$.  

Let us now consider a fixed $[x_0] \in P\big\slash G$. We may then define
\begin{equation*}
    i_{[x_0]}\colon \mathbb{R}\longrightarrow \mathbb{R}\times P\big\slash G,\quad t\longmapsto (t,\bar{\gamma}_{[x_0]}(t)),
\end{equation*}
and
\begin{equation*}
    \hat{i}_{[x_0]}\colon \mathbb{R}\times G\longrightarrow \mathbb{R}\times P,\quad (t,g)\longmapsto (t,\gamma_{x_0}(t)).
\end{equation*}
One easily verifies that the map ${\hat i}_{[x_0]}$ is $G$-equivariant, because, for an invariant vector field $Y$, $\gamma_{g_1 x_0} = \Psi_{g_1}\circ\gamma_{x_0}$ for $g_1\in G$. Moreover, since $\pi\circ\gamma_{x_0} = \bar{\gamma}_{[x_0]}$, we have that $[\hat{i}_{[x_0]}] = {i}_{[x_0]}$, when considered as maps $\mathbb{R} \to \mathbb{R} \times P\big\slash G$. 



The diagram
\begin{equation*}
\begin{CD}
\mathrm{T}(\mathbb{R}\times G)@>\mathrm{T}{\hat{i}}_{[x_0]}>>\mathrm{T}(\mathbb{R}\times P)\\
@VVV @VVV\\
\mathbb{R}\times G@>>{\hat{i}}_{[x_0]}>\mathbb{R}\times P.\\
\end{CD}
\end{equation*}
commutes. This leads to quotient maps in the corresponding commutative diagram on the quotient manifolds
\begin{equation*}
\begin{CD}
\mathrm{T}(\mathbb{R}\times G)\big\slash G@>[\mathrm{T}{\hat{i}}_{[x_0]}]>>\mathrm{T}(\mathbb{R}\times P)\big\slash G\\
@VVV @VVV\\
\mathbb{R}@>> [\hat{i}_{[x_0]}] = {i}_{[x_0]}>\mathbb{R}\times P\big\slash G.\\
\end{CD}
\end{equation*}

Recall that now $A=\mathrm{T}(\mathbb{R}\times G)\big\slash G$ and $p$ is a map $\mathbb{R}\times P\big\slash G \to\mathbb{R}$. From \Cref{prop4}, we know that we may identify the $p$-prolongation $\overline A$ of $A$ with the Atiyah Lie algebroid $\mathrm{T}(\mathbb{R}\times P)\big\slash G$, and from \Cref{prop:inducedsectionAtiyah}, we know that the $i_{[x_0]}$-induced section $I$ on this Atiyah Lie algebroid is  $[\mathrm{T}\hat{i}_{[x_0]}]$.

Next, we show how the reduction and reconstruction problem can be recast as a problem involving integral manifolds of EDS on manifolds and Lie algebroids.

Let us assume that the coordinates $(x^i,u^a)$ are adapted to the principal fiber bundle $P\to P\big\slash G$. With the help of a principal connection we may horizontally lift the coordinate vector fields $\partial/\partial x^i$ on $P\big\slash G$ to invariant horizontal vector fields $H_i$ on $P$. We can now write
\begin{equation*}
    Y=Y^i H_i + Y^a \frac{\partial }{\partial u^a}, \qquad \overline Y = Y^i \frac{\partial }{\partial x^i}.
\end{equation*}
We first consider the EDS of forms on the manifold $\mathbb{R}\times P$ defined by semi-basic forms (with respect to the fiber bundle $\pi\colon Q\to Q \big\slash G$), i.e., ${\mathcal I}^{\rm{sb}} = \langle {\mathrm d}x^i-Y^i \, {\mathrm d}t\rangle$. Since these forms are also invariant, we may reinterpret this as an EDS on the Atiyah Lie algebroid $\mathrm{T}({\mathbb R \times P})\big\slash G \to \mathbb R \times P\big\slash G$ (for the obvious extension of the action on $P$ to $\mathbb R \times P$), given by
\begin{equation*}
    {\widehat{\mathcal I}}^{\rm{sb}} = \langle \widehat{{\mathrm d}x^i}-\widehat{Y}^i \, \widehat{{\mathrm d}t}\rangle = \langle \delta x^i-{\widehat Y}^i \, \delta t\rangle,
\end{equation*}
where $\widehat{\,\star\,}$ means their interpretation as maps $\mathbb{R}\times P\big\slash G\to \mathrm{T}^{\ast}(\mathbb{R}\times P)\big\slash G$  (in accordance with the notation we had introduced at the beginning of \Cref{sec:atiyahalgebroid}), and $\delta$ is the exterior derivative on the Atiyah Lie algebroid. The algebraic ideal generated by this $1$-form is, in fact a differential ideal on the Lie algebroid (so that ${\widehat{\mathcal I}}^{\rm{sb}}= \langle \delta x^i-{\widehat Y}^i \, \delta t\rangle_{\mathrm{alg}}$), because
\begin{equation*}
    \delta \left(\delta x^i-{\widehat Y}^i \, \delta t\right)  = - \frac{\partial {\widehat Y}^i}{\partial x^j} \, \delta x^j \wedge\delta t =  \frac{\partial {\widehat Y}^i}{\partial x^j} \, \delta t \wedge \left(\delta x^j-{\widehat Y}^j \, \delta t\right).   
\end{equation*}
Now, fix $[x_0]=(x^i_0) \in P\big\slash G$. Consider the integral manifold $i_{[x_0] }\colon \mathbb R \to \mathbb R \times P\big\slash G, t\mapsto (t,x^i(t))$ of ${\widehat{\mathcal I}}^{\rm{sb}}$. Then $x^i(t)$ solves the ODEs of the reduced vector field: ${\dot x}^i=Y^i(x)$, with $x^i(0)=x^i_0$.

On $P$ (and on $\mathbb R \times P$), we may consider the forms $\xi$ with the property that $\xi(H)=0$, for all horizontal vector fields $H$. If $\Gamma^a_i(x,u)$ are connection coefficients, the set of all such forms is spanned by
\begin{equation*}
    {\mathrm d}u^a - \Gamma^a_i \, {\mathrm d}x^i.
\end{equation*}
We can add them to ${\mathcal I}^{\rm{sb}}$, to obtain a new EDS on $\mathbb R \times P$, with
\begin{equation*}
    {\mathcal I}^{h} = \langle   {\mathrm d}x^i-Y^i \, {\mathrm d}t,  {\mathrm d}u^a - \Gamma^a_i \, {\mathrm d}x^i\rangle =   \langle   {\mathrm d}x^i-Y^i \, {\mathrm d}t,  {\mathrm d}u^a - \Gamma^a_iY^i \, {\mathrm d}t\rangle .
\end{equation*}
The algebraic ideal generated by these forms is in fact a differential ideal, since
\begin{equation*}
    \mathrm{d}({\mathrm d}x^i-Y^i \, {\mathrm d}t) =  \frac{\partial Y^i}{\partial x^j}   {\mathrm d}t \wedge \left({\mathrm d}x^j-Y^j \, {\mathrm d}t\right),
\end{equation*}
\begin{equation*}
    {\mathrm d}\left( {\mathrm d}u^a - \Gamma^a_iY^i \, {\mathrm d}t \right) = \left( \frac{\partial (\Gamma^a_iY^i)}{\partial x^j} \right){\mathrm d}t \wedge \left({\mathrm d}x^j -Y^j \, {\mathrm d}t\right) +\left( \frac{\partial(\Gamma^a_iY^i)}{\partial u^b} \right){\mathrm d}t\wedge \left({\mathrm d}u^b - \Gamma^b_kY^k \, {\mathrm d}t\right). 
\end{equation*}
If we now fix $x_0\in P$ (with $\pi(x_0)=[x_0]$, and thus $x_0=(x^i_0,u^a_0)$), and if we suppose that $i^h_{x_0}\colon \mathbb R \to P, t \mapsto (x^i_h(t), u^a_h(t)) $ is an integral manifold of ${\mathcal I}^{h}$, then $(i^h_{x_0})^{\ast} ({\mathrm d}x^i-Y^i  {\mathrm d}t) = 0$ and $(i^h_{x_0})^{\ast} ({\mathrm d}u^a - \Gamma^a_i{\mathrm d}x^i) = (i^h_{x_0})^{\ast} ({\mathrm d}u^a - \Gamma^a_iY^i{\mathrm d}t)  = 0$. Moreover, $x^i_h(t)$ is a solution of the same initial value problem, and therefore $x^i_h(t)=x^i(t)$ is again a solution ${\bar\gamma}_{[x_0]}$ of the reduced equations. The functions $u^a(t)$ satisfy
\begin{equation*}
    {\dot u}^a = \Gamma^a_i(x,u){\dot x}^i,
\end{equation*}
which are the differential equations that determine the horizontal lift $\bar{\gamma}_{x_0}^H$.

Now, we consider the EDS that corresponds to the vector field $Y$ on $P$. It consists of the forms
\begin{equation*}
    {\mathcal I} = \langle  {\mathrm d}x^i-Y^i \, {\mathrm d}t,  {\mathrm d}u^a - ( \Gamma^a_iY^i + Y^a) \, {\mathrm d}t\rangle,
\end{equation*}
on $\mathbb R \times P$. Let $\overline{\Phi}_g\coloneqq \mathrm{id}_{\mathbb{R}}\times \Psi_g$, which is an action on $\mathbb{R}\times P$. We show that we can find integral manifolds $i_{x_0}\colon \mathbb R \to \mathbb R \times P$ of the type $i_{x_0}=\overline{\Phi}_{g(t)}\circ i^h_{x_0}$, where $g(t)$ is a curve in $G$ through $e$ (here we regard $i^h_{x_0}$ as the map $\mathbb{R}\to \mathbb{R}\times P$, $t\mapsto (t,x_h^i(t),u_h^a(t))$). We show that $g(t)$ is (again) determined by the reconstruction equation.

First, observe that, since $(\overline{\Phi}_{g(t)}\circ i^h_{x_0})^{\ast} = (i^h_{x_0})^{\ast}\circ \overline{\Phi}_{g(t)}^{\ast}$, and since the forms ${\mathrm d}x^i-Y^i \, {\mathrm d}t$ are invariant, we already know that
\begin{equation*}
    (i_{x_0})^{\ast} ({\mathrm d}x^i-Y^i \, {\mathrm d}t) = ((i^h_{x_0})^{\ast}\circ \overline{\Phi}_{g(t)}^{\ast}) ({\mathrm d}x^i-Y^i \, {\mathrm d}t) = (i^h_{x_0})^{\ast}({\mathrm d}x^i-Y^i \, {\mathrm d}t) = 0.
\end{equation*}
It therefore remains to compute $(i_{x_0})^{\ast}({\mathrm d}u^a - ( \Gamma^a_iY^i + Y^a) \, {\mathrm d}t)$. This determines a differential equation on $g(t)$, which corresponds to the reconstruction equation. It would lead us too far to do it in the current generality. Therefore, we will limit ourselves here to an illuminating example.

\begin{example} {\bf A radially symmetric system.}   Consider the system on the punctured plane $N=\mathbb{R}^2\setminus \{(0,0)\}$ given by
\begin{equation}\label{eqn:systemcartesian}
    \begin{cases} \dot{x}=x+y(x^2+y^2), \\ \dot{y}=y-x(x^2+y^2). \end{cases}
\end{equation}
This can be written as an EDS (in the sense of manifolds) on $\mathbb{R}\times N$ by
\begin{equation*}
    \mathcal{I}=\langle \theta^1,\theta^2\rangle=\langle \mathrm{d}x-(x+y(x^2+y^2))\, \mathrm{d}t,\mathrm{d}y-(y-x(x^2+y^2))\, \mathrm{d}t\rangle.
\end{equation*}
Note that $G=\SO(2)$ is a symmetry group of this system, with the usual action
\begin{equation*}
    \Theta\colon G\times N\longrightarrow N,\quad \left(\begin{pmatrix} \cos(s) & -\sin(s) \\ \sin(s) & \cos(s) \end{pmatrix},(x,y)\right)\longmapsto (\cos(s)x-\sin(s)y,\sin(s)x+\cos(s)y).
\end{equation*}
The Lie algebra $\mathfrak{g}$ of $G$ is generated by
\begin{equation*}
    \mathbf{e}_1=\begin{pmatrix} 0 & -1 \\ 1 & 0 \end{pmatrix}.
\end{equation*}
The infinitesimal generator corresponding to $\mathbf{e}_1$ is given by
\begin{equation*}
    Z_1=-y\frac{\partial}{\partial x}+x\frac{\partial}{\partial y}.
\end{equation*}
The coordinate $r\coloneqq \sqrt{x^2+y^2}$ is invariant, hence $r$ can be regarded as a coordinate of $N\big\slash G\simeq \mathbb{R}_{>0}$. We now wish to rewrite $\mathcal{I}$ in terms of invariant forms which generate $\mathcal{I}$, i.e., we would like to find two independent pairs $(f,g)$, such that the Lie derivative satisfies $\mathcal{L}_{Z_1}(f\theta^1+g\theta^2)=0$ (here $f,g$ are functions from $N$ to $\mathbb{R}$). One solution is given by $(f(x,y),g(x,y))=\left(\frac{x}{\sqrt{x^2+y^2}},\frac{y}{\sqrt{x^2+y^2}}\right)$, and another by $(f(x,y),g(x,y))=\left(-\frac{y}{x^2+y^2},\frac{x}{x^2+y^2}\right)$. We then see that
\begin{equation*}
    \mathrm{d}r-r \, \mathrm{d}t=\frac{x}{\sqrt{x^2+y^2}}\theta^1+\frac{y}{\sqrt{x^2+y^2}}\theta^2,\quad \mathrm{d}\theta+r^2 \, \mathrm{d}t=-\frac{y}{x^2+y^2}\theta^1+\frac{x}{x^2+y^2}\theta^2,
\end{equation*}
where
\begin{equation*}
    \mathrm{d}r=\frac{x}{\sqrt{x^2+y^2}} \, \mathrm{d}x+\frac{y}{\sqrt{x^2+y^2}} \, \mathrm{d}y,\quad \mathrm{d}\theta=-\frac{y}{x^2+y^2} \, \mathrm{d}x+\frac{x}{x^2+y^2} \, \mathrm{d}y.
\end{equation*}
Hence, we can rewrite $\mathcal{I}$ in terms of invariant forms:
\begin{equation*}
    \mathcal{I}=\langle \mathrm{d}r-r \, \mathrm{d}t,\mathrm{d}\theta+r^2 \, \mathrm{d}t\rangle,
\end{equation*}
which corresponds to the system
\begin{equation}\label{eqn:systempolar}
    \begin{cases} \dot{r}=r, \\ \dot{\theta}=-r^2, \end{cases}
\end{equation}
on $P=\mathbb{R}_{>0}\times \mathbb{S}^1$, whose corresponding vector field $Y$ is given by $r\frac{\partial}{\partial r}-r^2\frac{\partial}{\partial\theta}$. 
The general solution of the system is given by
\begin{equation} \label{eq:solrtheta}
    r(t)=r_0 e^t,\quad \theta(t)=\theta_0+\frac{1}{2}r_0^2 (1-e^{2t}),
\end{equation}
with $r_0>0$ and $\theta_0\in \mathbb{R}$. 

When we consider $\mathcal I$ as an EDS on the manifold ${\mathbb R}\times P$, we can interpret the above solution as an integral manifold of $\mathcal I$ for each fixed pair $(r_0,\theta_0)\in P$,
\begin{equation*}
    {i}_{(r_0,\theta_0)}\colon \mathbb R \longrightarrow \mathbb R \times P,\quad t \longmapsto (t,r(t),\theta(t)).
\end{equation*}
The group $G$ acts on $P$ by
\begin{equation*}
    \Psi\colon G\times P\longrightarrow P,\quad \left(\begin{pmatrix} \cos(s) & -\sin(s) \\ \sin(s) & \cos(s) \end{pmatrix},(r,\theta)\right)\longmapsto (r,\theta+s).
\end{equation*}


We now consider the Atiyah Lie algebroid $A=\mathrm{T}(\mathbb{R}\times P)\big\slash G$. On this Lie algebroid we consider the exterior differential system given by
\begin{equation*}\label{eqn:laedsradial}
    \widehat{\mathcal{I}}^{\mathrm{sb}}=\langle \widehat{\mathrm{d}r}-r \, \widehat{\mathrm{d}t}\rangle,
\end{equation*}
where, again, $\widehat{\mathrm{d}r},\widehat{\mathrm{d}t}\colon \mathbb{R}\times P\big\slash G\to \mathrm{T}^{\ast}(\mathbb{R}\times P)\big\slash G$ are the   forms on the Lie algebroid $\mathrm{T}(\mathbb{R}\times P)\big\slash G$ that one may construct from the  invariant forms $\mathrm{d}r,\mathrm{d}t$  on the manifold $\mathbb{R}\times P$. We conclude that, indeed, finding integral manifolds  of  $\widehat{\mathcal{I}}^{\mathrm{sb}}$ on the Lie algebroid  corresponds to finding a solution of the reduced equation,
\begin{equation*}
    \dot r =r.
\end{equation*}
If we fix  $r_0>0$, we find one such reduced solution, with corresponding integral manifold
\begin{equation*}
    i_{r_0} \colon \mathbb{R}\longrightarrow \mathbb{R}\times P\big\slash G,\quad   t\longmapsto (t,\bar{\gamma}_{r_0}(t)=r_0e^t).
\end{equation*}
We use the connection $\omega\colon \mathrm{T}P\to \mathfrak{g}$, defined by
\begin{equation*}
    \omega\left(\frac{\partial}{\partial r}\right)=\begin{pmatrix} 0 & 0 \\ 0 & 0 \end{pmatrix},\quad \omega\left(\frac{\partial}{\partial \theta}\right)=\mathbf{e}_1,
\end{equation*}
to compute the horizontal lift $\bar{\gamma}^H_{(r_0,\theta_0)}$ of $\bar{\gamma}_{r_0}$ to $(r_0,\theta_0)$. Since $\pi\circ \bar{\gamma}^H_{(r_0,\theta_0)}=\bar{\gamma}_{r_0}$, it follows that $\bar{\gamma}^H_{(r_0,\theta_0)}$ is of the form
\begin{equation*}
    \bar{\gamma}_{(r_0,\theta_0)}^H(t)=(r_0e^t,\theta_1(t)),
\end{equation*}
where $\theta_1$ is a smooth function determined by the condition $\omega\left(\dot{\bar{\gamma}}^H_{(r_0,\theta_0)}(t)\right)=0$.
We find that
\begin{equation*}
    \dot{\bar{\gamma}}_{(r_0,\theta_0)}^H(t)=(r_0e^t,\theta_1(t),r_0e^t,\dot{\theta}_1(t)),
\end{equation*}
hence $\omega\left(\dot{\bar{\gamma}}^H_{(r_0,\theta_0)}(t)\right)=0$ if and only if $\theta_1$ is constant, meaning $\theta_1(t)=\theta_0$.

A curve in $G$ with $g(0)=\begin{pmatrix} 1 & 0 \\ 0 & 1 \end{pmatrix}$ is of the form $g(t)=\begin{pmatrix} \cos \epsilon(t) & -\sin \epsilon(t) \\ \sin \epsilon(t) & \cos \epsilon (t) \end{pmatrix}$ for some smooth function $\epsilon(t)$ with $\epsilon(0)=0$. With that, $(g(t))^{-1}\dot{g}(t)=\dot{\epsilon}\, \mathbf{e}_1$.  Since $\omega(Y\circ \bar{\gamma}^H_{(r_0,\theta_0)})=-r_0^2e^{2t} \mathbf{e}_1$, the reconstruction problem is given by 
\begin{equation}\label{eq:recexample}\dot{\epsilon}=-r_0^2e^{2t}
\end{equation}
Its solution with the given initial condition is $\epsilon(t)=\frac{1}{2}r_0^2(1-e^{2t})$. 

Computing $\Psi_{g(t)} \bar{\gamma}_{(r_0,\theta_0)}^H(t)$ gives the solutions to \Cref{eqn:systempolar}, and hence to \Cref{eqn:systemcartesian}. Here, we recover the expressions given in \eqref{eq:solrtheta}.


We now return to the equation $(i_{x_0})^{\ast}({\mathrm d}u^a - ( \Gamma^a_iY^i + Y^a){\mathrm d}t)=0$, for this example. Here, the connection coefficients vanish, $\Gamma^a_i=0$. With this, $Y^a = -r^2$. Moreover, in view of the above calculations, we have $i^h_{(r_0,\theta_0)}\colon t \mapsto(t,r_0e^t,\theta_0)$. Therefore $i_{(r_0,\theta_0)}\colon t \mapsto(t,r_0e^t,\theta_0+\epsilon(t))$. We get
\begin{equation*}
    (i_{x_0})^{\ast}({\mathrm d}u^a - ( \Gamma^a_iY^i + Y^a){\mathrm d}t)= (i_{(r_0,\theta_0)})^{\ast}({\mathrm d}\theta  +r^2{\mathrm d}t) = ({\dot\epsilon}+
 r_0^2e^{2t}){\mathrm d}t.
\end{equation*}
Clearly, setting this to zero gives the same reconstruction equation as in \eqref{eq:recexample}.





\end{example}

\section{An application to the invariant inverse problem}\label{sec:invariantinverse}
In this section, we deal with the following problem, which we will call the \emph{invariant inverse problem}: \emph{Given a (possibly time-dependent) $G$-invariant second-order system on a Lie group $G$, when does a (possibly time-dependent) $G$-invariant regular Lagrangian $L$ whose Euler--Lagrange equations are equivalent to this system exist?}

Some progress on this problem has been done in the time-independent case in \cite{barbero2016invproblemliealgebroids,crampin2008invariantinv,muzsnay2005invariantinv}. However, contrary to the original inverse problem for Lagrangian mechanics, exterior differential systems have never been used to solve this problem. We will show that the invariant inverse problem in the time-dependent case fits within the framework of EDS on Lie algebroids.

\subsection{The reduced Helmholtz conditions}

Let $n=\dim(G)$. Suppose that we are given a (time-dependent) regular invariant Lagrangian $L\in C^{\infty}(\mathbb{R}\times \mathrm{T}G)$. Its restriction to $\mathbb{R}\times \mathrm{T}_e G\simeq \mathbb{R}\times \mathfrak{g}$ is a function $l\in C^{\infty}(\mathbb{R}\times \mathfrak{g})$. Then, as proven in \cite{marsden1999eulerpoincare}, finding a solution $g(t)\in G$ of the Euler--Lagrange equations is equivalent to finding a solution $w(t)\in \mathfrak{g}$ of the Euler--Poincar{\'e} equations, given by
\begin{equation*}
    \frac{\mathrm{d}}{\mathrm{d}t}\left(\frac{\partial l}{\partial w^j}\right)=C_{ij}^k \frac{\partial l}{\partial w^k}w^i,
\end{equation*}
where $C_{ij}^k$ are the structure constants of $\mathfrak{g}$ corresponding to a basis used to define the coordinates. If we are given a solution $w(t)$ to the Euler--Poincar{\'e} equations, a solution to the Euler--Lagrange equations can be obtained by solving the reconstruction equation $g(t)^{-1}\dot{g}(t)=w(t)$ \cite{mestdag2008lagrangepoincare}.

Running a bit ahead of the preliminaries that we will introduce in the next subparagraph, we already make the following definition:

\begin{definition}\label{def:reducedmultiplier}
Let $\gamma\in \mathcal{X}(\mathbb{R}\times \mathfrak{g})$ be as in \eqref{smallgamma}, and let $(k_{ij})$ be a matrix of functions on $\mathbb{R}\times \mathfrak{g}$. We say that $(k_{ij})$ is a \emph{reduced multiplier matrix} for $\gamma$ if it satisfies the following conditions:
\begin{equation*}
    \det(k_{ij})\neq 0,\quad k_{ij}=k_{ji},
\end{equation*}
\begin{equation*}
    \gamma(k_{ij})-k_{kj}\lambda_i^k-k_{ik}\lambda_j^k=0,
\end{equation*}
\begin{equation*}
    k_{ki}\phi_j^k=k_{kj}\phi_i^k,
\end{equation*}
\begin{equation*}
    \frac{\partial k_{jk}}{\partial w^i}=\frac{\partial k_{ik}}{\partial w^j}.
\end{equation*}
We call the above equations in $k_{ij}$ the \emph{reduced Helmholtz conditions}.
\end{definition}

In contrast to the original inverse problem, where the existence of a multiplier matrix is both necessary and sufficient to ensure the existence of a Lagrangian, the relationship between reduced multipliers and invariant Lagrangians is more complicated. While the existence of a reduced multiplier matrix is sufficient to guarantee the existence of a Lagrangian, the Lagrangian may not necessarily be invariant. To determine whether or not an invariant Lagrangian exists, an additional condition is needed.


The next proposition extends a similar statement for the autonomous case in \cite{crampin2008invariantinv} to the time-dependent case.

\begin{proposition}\label{prop:cohomology} A (time-dependent) reduced multiplier matrix $k_{ij}$ for $\gamma$ determines, for each value of $t$, a cohomology class  in  $H^2(\mathfrak{g})$. If these classes vanish, the multiplier matrix also determines, for each value of $t$, a cohomology class  in $H^1(\mathfrak{g})$. If they also vanish, the vector field $\gamma$ is derivable from an invariant Lagrangian. \\
Conversely, if the vector field $\gamma$ is derivable from an invariant Lagrangian all the above  cohomology classes vanish.
\end{proposition}

\begin{proof} For this proof, we use the same notation and terminology as in Thm.~4 of \cite{crampin2008invariantinv} and we limit ourselves here to an adjustment of their reasoning from the autonomous to the time-dependent case.

We denote the Euler--Poincar\'e expressions by
\begin{equation*}
    V_i\coloneqq \gamma\left(\frac{\partial l}{\partial w^j}\right)-C_{ij}^k \frac{\partial l}{\partial w^k}w^i.
\end{equation*}
 If we assume that the reduced Helmholtz conditions are satisfied, then the same reasoning as in \cite{crampin2008invariantinv} shows that the expressions $V_i$ of any function $l$ whose Hessian is the multiplier $k_{ij}$ are of type $V_i=\mu_{ji}w^j+\nu_i$ for some functions $\mu_{ij}(t)$ and $\nu_i(t)$ that, moreover, satisfy
\begin{equation} \label{coh}
    {\dot\mu}_{ij} + C^l_{ij}\nu_l = 0,\qquad  \sum_{{\rm{cyclic}}(i,j,k)} \mu_{il}C^l_{jk} = 0.
\end{equation}
We have the freedom to change $l$ into $l'=l+\theta_k(t) w^k$, and keep the same Hessian. In that case
\begin{equation*}
    \nu_i'=\nu_i +{\dot\theta}_i, \qquad \mu_{ij}'=\mu_{ij}-\theta_k C^k_{ij}.
\end{equation*}
One easily checks that the conditions \eqref{coh} are invariant under the above changes. 


We conclude that, in order for $V_i \equiv 0$, we need to find $\theta_i(t)$ such that both $\nu_i'=0$ and $\mu_{ij}'=0$. This will be the case when
\begin{equation*} \label{timecoh} 
\left\{\begin{array}{l}{\dot\mu}_{ij}  + C^l_{ij}\nu_l = 0, \\ \sum_{{\rm{cyclic}}(i,j,k)} \mu_{il}C^l_{jk} = 0.\end{array} \right. \qquad \iff \qquad \mbox{$\exists\, \theta_k(t)$ s.t. } \left\{\begin{array}{l} \nu_i = - {\dot\theta}_i, \\ \mu_{ij}=\theta_ k C^k_{ij}. \end{array} \right.
\end{equation*}
We may regard $t$ as a parameter. The last equation says that the multiplier $(k_{ij})$ determines, for each $t$, a cohomology class $(\mu_{ij})$ in $H^2(\mathfrak{g})$. If we assume that it vanishes for each $t$, there exists $\theta_k(t)$ such that $\mu_{ij}(t) = \theta_k(t)C^k_{ij}$. The first equation then becomes $( {\dot\theta}_l +\nu_l )C^l_{ij} = 0$. This means that, under the first assumption,  we again get for each $t$ a cohomology class $\xi_l={\dot\theta}_l +\nu_l$ in $H^1(\mathfrak{g})$. If this class also vanishes for each $t$, then we have  $\nu_i = - {\dot\theta}_i$, as required.
\end{proof}

In the remainder of this paper, it is our goal to show that the existence of a reduced multiplier matrix is equivalent to finding integral manifolds of some differential ideal on a Lie algebroid.

\subsection{The  IP Lie algebroid}\label{subsec:liealgebroidinverse}

Let $\{E_i\}$ be a basis of the Lie algebra of $\mathfrak{g}$. By left-translating this basis, we obtain a left-invariant basis $\{\widehat{E}_i\}$ of $\mathcal{X}(G)$. Similarly, denote by $\{\widetilde{E}_i\}$ the set of vector fields in $\mathcal{X}(G)$ obtained by right-translating the basis $\{E_i\}$. The bases are related by
\begin{equation*}
    \widehat{E}_i(g)=A_i^j(g)\widetilde{E}_j(g),
\end{equation*}
where $(A_i^j(g))$ is the matrix representation of the adjoint representation $\Ad_g$. In particular, this matrix satisfies $A_i^j(e)=\delta_i^j$, where $e\in G$ is the identity element. We now identify the Lie algebra with the set of left-invariant vector fields. Then
\begin{equation*}
    [\widehat{E}_i,\widehat{E}_j]=C_{ij}^k \widehat{E}_k,\quad [\widetilde{E}_i,\widetilde{E}_j]=-C_{ij}^k \widetilde{E}_k.
\end{equation*}
In the following, we denote by $(w^i)$ the coordinates of a vector $v_g\in \mathrm{T}_g G$ with respect to the basis $\{\widehat{E}_i\}$, i.e.,
\begin{equation*}
    v_g=w^i \widehat{E}_i(g).
\end{equation*}

The vector field $Z=B\frac{\partial}{\partial t}+\Xi^j \widehat{E}_j^C+F^j \widehat{E}_j^V$ on $\mathbb{R}\times \mathrm{T}G$ is left-invariant if and only if $[\widetilde{E}_i^C,Z]=0$ for all $i\in \{1,\dots,n\}$. By left-invariance of $\partial/\partial t$ and $\widehat{E}_j$, one can observe that
\begin{equation*}
    [\widetilde{E}_i^C,Z]=\widetilde{E}_i^C(B)\frac{\partial}{\partial t}+\widetilde{E}_i^C(\Xi^j)\widehat{E}_j^C+\widetilde{E}_i^C(F^j)\widehat{E}_j^V.
\end{equation*}
Since $\{\frac{\partial}{\partial t},\widehat{E}_j^C, \widehat{E}_j^V\}$ is a basis for the vector fields on $\mathbb{R}\times \mathrm{T}G$, it follows that $Z$ is left-invariant if and only if
\begin{equation*}
    \widetilde{E}_i^C(B)=\widetilde{E}_i^C(\Xi^j)=\widetilde{E}_i^C(F^j)=0,
\end{equation*}
which means that $B$, $\Xi^j$ and $F^j$ are invariant functions. These functions therefore project to functions $b$, $\xi^j$ and $f^j$ on $(\mathbb{R}\times \mathrm{T}G)\big\slash G\simeq \mathbb{R}\times \mathfrak{g}$, which satisfy
\begin{equation*}
    B=b\circ \pi, \quad \Xi^j=\xi^j\circ \pi,\quad F^j=f^j\circ \pi,
\end{equation*}
where $\pi\colon \mathbb{R}\times \mathrm{T}G\to \mathbb{R}\times \mathfrak{g}$ is the natural projection.

A set $\{\xi^j\}$ of $n=\dim(\mathfrak{g})$ functions on $\mathbb{R}\times \mathfrak{g}$ can be interpreted in two equivalent ways. First, the elements of the set could be viewed as the coefficients of a $C^{\infty}(\mathbb{R}\times \mathfrak{g},\mathfrak{g})$-map, namely the map
\begin{equation*}
    \xi\colon (t,w)\longmapsto \xi^i(t,w)E_i.
\end{equation*}
Using this interpretation, we define $\mathfrak{g}$-valued functions $\xi$ and $\eta$ from $\{\xi^j\}$ and $\{\eta^j\}$ respectively. Second, they can be viewed as components of the vector field $\bar{\xi}=\xi^j \partial/\partial w^j$ on $\mathbb{R}\times \mathfrak{g}$.

A time-dependent second-order differential equation field on $G$ can be written in the form
\begin{equation*}
    \Gamma=\frac{\partial}{\partial t}+w^i \widehat{E}_i^C+\Gamma^i \widehat{E}_i^V.
\end{equation*}
If $\Delta$ is the identity map in $C^{\infty}(\mathbb{R}\times \mathfrak{g},\mathbb{R}\times \mathfrak{g})$, then an invariant $\Gamma$ reduces to the section $(\Delta,\gamma)$ of $\mathbb{R}\times \mathfrak{g}\times \mathrm{T}(\mathbb{R}\times \mathfrak{g})\to \mathbb{R}\times \mathfrak{g}$, where
\begin{equation}\label{smallgamma}
    \gamma=\frac{\partial}{\partial t}+\gamma^i \frac{\partial}{\partial w^i},
\end{equation}
is called the \emph{reduced vector field}.

We are now ready to define the Lie algebroid of interest. We first define its vector bundle $A\to M$ to be
\begin{equation}\label{eqn:liealgebroidinvprob}
    A=\mathbb{R}\times \mathfrak{g}\times \mathbb{R}\times \mathfrak{g}\times \mathfrak{g}\longrightarrow M=\mathbb{R}\times \mathfrak{g}.
\end{equation}
Its sections are given by maps of the form
\begin{equation*}
    (t,w)\longmapsto (t,w,b(t,w),\xi(t,w),\eta(t,w)).
\end{equation*}
A basis of sections of this bundle is given by $\{T_0,e_i,W_i\}$, where
\begin{equation*}
    T_0\colon (t,w)\longmapsto (t,w,1,0,0),
\end{equation*}
\begin{equation*}
    e_i\colon (t,w)\longmapsto (t,w,0,E_i,0),\quad W_i\colon (t,w)\longmapsto (t,w,0,0,E_i).
\end{equation*}
We claim that this vector bundle also admits the structure of a Lie algebroid. We define an anchor $\rho$ by
\begin{equation*}
    \rho(b,\xi,X)=b\frac{\partial}{\partial t}+\mathcal{A}(\bar{\xi})+X\in \mathcal{X}(M),
\end{equation*}
where $\mathcal{A}$ is a type $(1,1)$-tensor field on $M$ given by
\begin{equation*}
    \mathcal{A}=w^k C_{kj}^l \frac{\partial}{\partial w^l}\otimes \mathrm{d}w^j.
\end{equation*}
The action of $\rho$ on the sections $\{T_0,e_i,W_i\}$ of $A\to M$ is given by
\begin{equation}\label{IPalganchor}
    \rho(T_0)=\frac{\partial}{\partial t},\quad \rho(e_i)=w^k C_{ki}^j \frac{\partial}{\partial w^j},\quad \rho(W_i)=\frac{\partial}{\partial w^i}.
\end{equation}
Taking inspiration from the identities (cf. \Cref{subsec:evolution})
\begin{equation*}
    [\widehat{E}_i^C,\widehat{E}_j^C]=C_{ij}^k \widehat{E}_k^C,\quad [\widehat{E}_i^C,\widehat{E}_j^V]=C_{ij}^k \widehat{E}_k^V,\quad [\widehat{E}_i^V,\widehat{E}_j^V]=0,
\end{equation*}
we define a bracket given by
\begin{equation} \label{IPalgbracket1}
    \llbracket T_0,e_i\rrbracket=0,\quad \llbracket T_0,W_i\rrbracket=0,
\end{equation}
\begin{equation} \label{IPalgbracket2}
    \llbracket e_i,e_j\rrbracket=C_{ij}^k e_k,\quad \llbracket e_i,W_j\rrbracket=C_{ij}^k W_k,\quad \llbracket W_i,W_j\rrbracket=0.
\end{equation}

\begin{definition} \label{def:IPalgebroid}
The \emph{IP Lie algebroid} is the Lie algebroid that is defined on the vector bundle \eqref{eqn:liealgebroidinvprob} by means of the anchor \eqref{IPalganchor} and the bracket relations \eqref{IPalgbracket1} and \eqref{IPalgbracket2}.
\end{definition}

The  IP Lie algebroid we have introduced above is a generalization (from the autonomous to the time-dependent context) of a Lie algebroid that already appeared in \cite{crampin2008invariantinv} (see its Section 7, for an intrinsic definition). There, it was pointed out that this current Lie algebroid is not a \emph{prolongation Lie algebroid}, although such a Lie algebroid is often used in the context of Lagrangian systems on Lie algebroids (see e.g.,  \cite{martinez2001lagrangeliealgebroid}). Typically, the prolongation Lie algebroid satisfies the bracket relation $\llbracket e_i,W_j\rrbracket=0$, whereas in our case we have $\llbracket e_i,W_j\rrbracket=C_{ij}^k W_k$. Although $A$ itself is not a prolongation Lie algebroid, we will nevertheless make use of a prolongation when we consider the Lie algebroid $\overline A$ of \Cref{sec:pullback}.

One can verify that the horizontal lift of $\widehat{E}_i$ satisfies
\begin{equation*}
    \widehat{E}_i^H=\widehat{E}_i^C-\Lambda_i^j\widehat{E}_j^V,\quad \Lambda_i^j\coloneqq \frac{1}{2}\left(w^k C_{ki}^j-\widehat{E}_i^V(\Gamma^j)\right).
\end{equation*}
Hence, we define sections $\Gamma_0\coloneqq T_0+w^j e_j+\gamma^j W_j$ and $H_i\coloneqq e_i-\lambda_i^j W_j$, where
\begin{equation*}
    \lambda_i^j\coloneqq
    \frac{1}{2}\left(w^k C_{ki}^j-\frac{\partial \gamma^j}{\partial w^i}\right).
\end{equation*}
The sections $\Gamma_0,H_i,W_i$ form a basis of sections of $A\to M$. The anchor map on these sections satisfy
\begin{equation*}
    \rho(\Gamma_0)=\gamma,\quad \rho(H_i)=\psi_i^j \frac{\partial}{\partial w^j},
\end{equation*}
where
\begin{equation*}
    \psi_i^j\coloneqq w^k C_{ki}^j-\lambda_i^j=\frac{1}{2}\left(w^k C_{ki}^j+\frac{\partial \gamma^j}{\partial w^i}\right).
\end{equation*}
After some computation, we can obtain expressions for the brackets on this new basis (compare with the expressions in \Cref{prop:bracketsstandard1} and \Cref{prop:bracketsstandard2}).
\begin{lemma}\label{lem:brackets} On the  IP Lie algebroid, we have that
\begin{equation*}
    \llbracket \Gamma_0,W_i\rrbracket=-H_i+\lambda_i^j W_j,\quad \llbracket \Gamma_0,H_i\rrbracket=\lambda_i^j H_j+\phi_i^j W_j,
\end{equation*}
\begin{equation*}
    \llbracket H_i,W_j\rrbracket=(C_{ij}^k+\lambda_{ij}^k) W_k,\quad \llbracket H_i,H_j\rrbracket=C_{ij}^k H_k+r_{ij}^k W_k,
\end{equation*}
where
\begin{equation*}
    \phi_i^j\coloneqq -w^k C_{ki}^l \frac{\partial \gamma^j}{\partial w^l}-\gamma^k C_{ik}^j-\gamma(\lambda_i^j)-\lambda_i^k\lambda_k^j,\qquad \lambda_{ij}^k\coloneqq \frac{\partial \lambda_i^k}{\partial w^j},
\end{equation*}
\begin{equation*}
    r_{ij}^k\coloneqq \left(w^l C_{lj}^m \lambda_{im}^k-w^l C_{li}^m \lambda_{jm}^k+\lambda_i^l \lambda_{jl}^k-\lambda_j^l \lambda_{il}^k\right)+C_{ij}^l \lambda_l^k+\lambda_i^l C_{jl}^k-\lambda_j^l C_{il}^k.
\end{equation*}
In particular, $r_{ij}^k=-r_{ji}^k$.
\end{lemma}

The above expressions for $\lambda^j_i$ and $\phi^j_i$ are the ones we need in \Cref{def:reducedmultiplier}.

\subsection{The two-form}

In this subsection, we prove the following result:
\begin{theorem}\label{thm:twoform}
A reduced multiplier matrix for $\gamma$ exists if and only if there exists a two-form $\Omega\in \Lambda^2(A)$ such that
\begin{enumerate}
    \item $\Omega$ is of maximal rank, in the sense that the $n$-fold wedge product $\Omega\wedge \dots\wedge \Omega$ is nowhere vanishing, 
    \item $\Omega(W_i,W_j)=0$,
    \item $\iota_{\Gamma_0} \Omega=0$,
    \item $\delta \Omega=0$.
\end{enumerate}
\end{theorem}
To this end, it is first helpful to define some $1$-forms on $A$. Let $\{T^0,e^i,W^i\}\subset \Lambda^1(A)$ be the dual basis to $\{T_0,e_i,W_i\}$. Now let $\{\Gamma^0,\Psi^i,\Theta^i\}\subset \Lambda^1(A)$ denote the dual basis of $\{\Gamma_0,W_i,H_i\}$. A simple computation shows that
\begin{equation*}
    \Gamma^0=T^0,\quad \Psi^i=W^i-\gamma^i T^0+\lambda_j^i \Theta^j,\quad \Theta^i=e^i-w^i T^0.
\end{equation*}
\begin{proposition}
If $\Omega$ satisfying the conditions in \Cref{thm:twoform} exists, then it is of the form $\Omega=k_{ij}\Psi^i\wedge \Theta^j$, where $\det(k_{ij})\neq 0$.
\end{proposition}
\begin{proof}
This proof is based on similar ideas as in \cite{crampin1984geohelmholtz}, but adapted to the current context of Lie algebroids. Since $\Omega(W_i,W_j)=0$ and $\iota_{\Gamma_0} \Omega=0$, it follows that $\Omega$ has the form
\begin{equation*}
    \Omega=k_{ij}\Psi^i\wedge \Theta^j+l_{ij}\Theta^i\wedge \Theta^j,
\end{equation*}
where we assume without loss of generality that $l_{ij}=-l_{ji}$. Since $\delta \Omega=0$ and $\iota_{\Gamma_0}\Omega=0$, it follows by Cartan's formula that
\begin{equation*}
    \mathcal{L}_{\Gamma_0}\Omega=\delta(\iota_{\Gamma_0}\Omega)+\iota_{\Gamma_0}(\delta \Omega)=0.
\end{equation*}
Since $\mathcal{L}_{\Gamma_0}\Psi^i(\Gamma_0)=0$, $\mathcal{L}_{\Gamma_0}\Psi^i(H_j)=-\phi_j^i$ and $\mathcal{L}_{\Gamma_0}\Psi^i(W_j)=-\lambda_j^i$, it follows that
\begin{equation*}
    \mathcal{L}_{\Gamma_0}\Psi^i=-\phi_j^i\Theta^j-\lambda_j^i\Psi^j.
\end{equation*}
Similarly,
\begin{equation*}
    \mathcal{L}_{\Gamma_0}\Theta^i=\Psi^i-\lambda_j^i \Theta^j.
\end{equation*}
By the Leibniz rule for Lie derivatives and skew-symmetry of $l_{ij}$, it follows that
\begin{multline*}
    \mathcal{L}_{\Gamma_0}\Omega=\rho(\Gamma_0)(k_{ij})\Psi^i\wedge \Theta^j+k_{ij}\left((\mathcal{L}_{\Gamma_0}\Psi^i)\wedge \Theta^j+\Psi^i\wedge (\mathcal{L}_{\Gamma_0}\Theta^j)\right)\\+\rho(\Gamma_0)(l_{ij})\Theta^i\wedge \Theta^j+2l_{ij}(\mathcal{L}_{\Gamma_0}\Theta^i)\wedge \Theta^j,
\end{multline*}
which gives
\begin{equation*}
    \mathcal{L}_{\Gamma_0}\Omega=k_{ij}\Psi^i\wedge \Psi^j+\left[\gamma(k_{ij})+2l_{ij}-\lambda_i^k k_{kj}-\lambda_j^k k_{ik}\right]\Psi^i\wedge \Theta^j+\left[\gamma(l_{ij})-k_{kj}\phi_i^k-2l_{kj}\lambda_i^k \right]\Theta^i\wedge \Theta^j.
\end{equation*}
Hence $\mathcal{L}_{\Gamma_0}\Omega=0$ if and only if the following three conditions hold:
\begin{enumerate}[label=(\alph*)]
    \item $k_{ij}=k_{ji}$,
    \item $2\gamma(l_{ij})-k_{kj}\phi_i^k-2l_{kj}\lambda_i^k+k_{ki}\phi_j^k+2l_{ki}\lambda_j^k=0$,
    \item $\gamma(k_{ij})+2l_{ij}-\lambda_i^k k_{kj}-\lambda_j^k k_{ik}=0$.
\end{enumerate}
Taking the skew-symmetric part of (c) gives that $l_{ij}=0$ (using (a)). This shows that $\Omega$ is of the form
\begin{equation*}
    \Omega=k_{ij}\Psi^i\wedge \Theta^j.
\end{equation*}
Finally, since $\Omega$ is of maximal rank, $(k_{ij})$ must be a non-singular matrix.
\end{proof}
It still remains to show that if $\Omega$ is of the form $k_{ij}\Psi^i\wedge \Theta^j$, where $\det(k_{ij})\neq 0$, then $\delta \Omega=0$ if and only  if the reduced Helmholtz conditions are satisfied. We must therefore compute its application on the triples $(\Gamma_0,W_i,W_j)$, $(\Gamma_0,W_i,H_j)$, $(\Gamma_0,H_i,H_j)$, $(W_i,W_j,W_k)$, $(W_i,W_j,H_k)$, $(W_i,H_j,H_k)$, $(H_i,H_j,H_k)$. We claim that four of these triples give the reduced Helmholtz conditions directly, two of them give redundant reduced Helmholtz conditions, and that one of them is identically satisfied.

A simple computation shows that
\begin{equation*}
    \Omega(\Gamma_0,H_i)=0,\quad \Omega(\Gamma_0,W_i)=0,\quad \Omega(W_i,W_j)=0,
\end{equation*}
\begin{equation*}
    \Omega(W_i,H_j)=k_{ij},\quad \Omega(H_i,H_j)=0,
\end{equation*}

We will make use of the following identity often:
\begin{equation*}
    \delta \omega(X,Y,Z)=\sum_{\mathrm{cyclic}(X,Y,Z)} \rho(X)(\omega(Y,Z))-\omega(\llbracket X,Y\rrbracket,Z),
\end{equation*}
where $\omega\in \Lambda^2(A)$ is a $2$-form and $X,Y,Z$ are sections of $A\to M$. From this, it is clear that $\Omega(W_i,W_j,W_k)=0$ holds identically.

\begin{proposition}
Let $\Omega=k_{ij}\Psi^i\wedge \Theta^j$. Then
\begin{equation*}
    \delta \Omega(\Gamma_0,W_i,W_j)=k_{ij}-k_{ji},
\end{equation*}
\begin{equation*}
    \delta \Omega(\Gamma_0,W_i,H_j)=\gamma(k_{ij})-k_{kj} \lambda_i^k-k_{ik}\lambda_j^k,
\end{equation*}
\begin{equation*}
    \delta \Omega(\Gamma_0,H_i,H_j)=k_{ki}\phi_j^k-k_{kj}\phi_i^k,
\end{equation*}
\begin{equation*}
    \delta \Omega(W_i,W_j,H_k)=\frac{\partial k_{jk}}{\partial w^i}-\frac{\partial k_{ik}}{\partial w^j}.
\end{equation*}
\end{proposition}
\begin{proof}
Note that
\begin{align*}
    \delta \Omega(\Gamma_0,W_i,W_j)&=\rho(\Gamma_0)(\Omega(W_i,W_j))+\rho(W_i)(\Omega(W_j,\Gamma_0))+\rho(W_j)(\Omega(\Gamma_0,W_i))\\&-\Omega(\llbracket \Gamma_0,W_i\rrbracket,W_j)-\Omega(\llbracket W_i,W_j\rrbracket,\Gamma_0)-\Omega(\llbracket W_j,\Gamma_0\rrbracket, W_i)\\&=-\Omega(\llbracket \Gamma_0,W_i\rrbracket,W_j)+\Omega(\llbracket \Gamma_0,W_j\rrbracket,W_i)\\&=-\Omega(-H_i,W_j)+\Omega(-H_j,W_i)\\&=-\Omega(W_j,H_i)+\Omega(W_i,H_j)\\&=-k_{ji}+k_{ij},
\end{align*}
and
\begin{align*}
    \delta \Omega(\Gamma_0,W_i,H_j)&=\rho(\Gamma_0)(\Omega(W_i,H_j))-\Omega(\llbracket \Gamma_0,W_i\rrbracket,H_j)-\Omega(\llbracket H_j,\Gamma_0\rrbracket,W_i)\\&=\gamma(k_{ij})-\lambda_i^k \Omega(W_k,H_j)+\lambda_j^k\Omega(H_k,W_i)\\&=\gamma(k_{ij})-k_{kj} \lambda_i^k-k_{ik}\lambda_j^k.
\end{align*}
Similarly,
\begin{equation*}
    \delta \Omega(\Gamma_0,H_i,H_j)=-\Omega(\llbracket \Gamma_0,H_i\rrbracket,H_j)-\Omega(\llbracket H_j,\Gamma_0\rrbracket,H_i)=-k_{kj}\phi_i^k +k_{ki}\phi_j^k,
\end{equation*}
and
\begin{equation*}
    \delta \Omega(W_i,W_j,H_k)=\rho(W_i)(\Omega(W_j,H_k))+\rho(W_j)(\Omega(H_k,W_i))=\frac{\partial k_{jk}}{\partial w^i}-\frac{\partial k_{ik}}{\partial w^j}. \qedhere
\end{equation*}
\end{proof}

\begin{proposition}\label{prop:redundant}
If $\Gamma_0\iprod \delta \Omega=0$, $\delta\Omega(W_i,W_j,H_k)=0$, $\delta \Omega(W_i,W_j,W_k)=0$ for all $i,j,k$ then $\delta \Omega(W_i,H_j,H_k)=0$ and $\delta {\Omega}(H_i,H_j,H_k)=0$ for all $i,j,k$.
\end{proposition}

\begin{proof}
One may constitute a proof that is similar to the one of the corresponding property in  \cite{do2016invprob}, by adapting it to the current context of Lie algebroids and its exterior derivative.
 \end{proof}

We have shown that the conditions $\delta\Omega(\Gamma_0,W_i,W_j)=\delta \Omega(\Gamma_0,W_i,H_j)=\delta \Omega(\Gamma_0,H_i,H_j)=\delta \Omega(W_i,W_j,H_k)=0$ are precisely the reduced Helmholtz conditions, and that the two conditions $\delta \Omega(W_i,H_j,H_k)=0$ and $\delta\Omega(H_i,H_j,H_k)=0$ hold provided that the other four conditions hold. A tedious calculation shows that the first of these `redundant' conditions is given explicitly by
\begin{equation*}
    -\psi_j^l \frac{\partial k_{ik}}{\partial w^l}+\psi_k^l \frac{\partial k_{ij}}{\partial w^l}+(C_{ji}^l+\lambda_{ji}^l)k_{lk}+C_{jk}^lk_{il}-(C_{ki}^l+\lambda_{ki}^l)k_{lj}=0,
\end{equation*}
which can be rewritten as (swapping the indices $k$ and $l$, and using skew-symmetry of the structure constants)
\begin{equation}\label{eqn:redundanthelmholtz1}
    \psi_l^k \frac{\partial k_{ij}}{\partial w^k}+\frac{1}{2}\left(k_{ki}C_{jl}^k+k_{kj}C_{il}^k-k_{kl}\frac{\partial^2 \gamma^k}{\partial w^i \partial w^j}\right)=\psi_j^k \frac{\partial k_{il}}{\partial w^k}+\frac{1}{2}\left(k_{ik} C_{lj}^k+k_{kl} C_{ij}^k-k_{kj}\frac{\partial^2 \gamma^k}{\partial w^i \partial w^l}\right),
\end{equation}
and the second by
\begin{equation}\label{eqn:redundanthelmholtz2}
    r_{ij}^l k_{lk}+r_{jk}^l k_{li}+r_{ki}^l k_{lj}=0.
\end{equation}

\subsection{The EDS algorithm}

\subsubsection{Introduction}\label{introsubsec}

By \Cref{thm:twoform}, the problem that we need to solve is a special case of the following problem: \emph{Suppose that $\tau\colon A\to M$ is the  IP Lie algebroid and let $\Sigma$ be a subspace of $\Lambda^2(A)$. Find all maximal rank closed two-forms in $\Sigma$.} It turns out that we can translate this problem into finding integral manifolds of a set of one-forms on a $p$-prolongation of $A$. We outline this in a two-step process, below. The ideas in this subsection generalize those of \cite{aldridge2003invprob,anderson1992invprob,do2016invprob}.

\textbf{Step 1:} Since $\langle \Sigma\rangle_{\mathrm{alg}}$ is, in general, not a differential ideal, we instead consider the maximal subspace $\Sigma'$ of $\Sigma$ such that it becomes a differential ideal. Indeed, if $\omega\in \Sigma$ is a closed form, then $\langle \omega\rangle_{\mathrm{alg}}$ is a differential ideal. Hence to find closed two-forms in $\Sigma$, it suffices to consider only forms in $\Sigma'$.

This maximal subspace can be defined as follows: If $\widetilde{\Sigma}_1$ and $\widetilde{\Sigma}_2$ are subspaces of $\Sigma$ such that both $\langle \widetilde{\Sigma}_1\rangle_{\mathrm{alg}}$ and $\langle \widetilde{\Sigma}_2\rangle_{\mathrm{alg}}$ are differential ideals, then their ideal generated by their sum, i.e., $\langle \widetilde{\Sigma}_1+\widetilde{\Sigma}_2\rangle_{\mathrm{alg}}$ is also a differential ideal. If $\Sigma'$ is the sum of all subspaces of $\Sigma$ that generate differential ideals, then $\Sigma'$ is the maximal subspace of $\Sigma$ such that $\langle \Sigma'\rangle_{\mathrm{alg}}$ is a differential ideal.

There exists a recursive method to construct $\Sigma'$ \cite[\S 5]{anderson1992invprob}. Suppose we let $\Sigma^0\coloneqq \Sigma$. Then the process produces a sequence of submodules $\Sigma^0\supset \Sigma^1\supset \Sigma^2\supset \dots\supset \Sigma'$. Obtaining $\Sigma^{N+1}$ from $\Sigma^N$ involves computing the Lie algebroid exterior derivative of forms $\omega\in \Sigma^N$, and then checking whether or not $\delta \omega\in \langle \Sigma^N\rangle_{\mathrm{alg}}$ for all such $\omega$. If it holds, then $\Sigma^N=\Sigma'$. Otherwise, this leads to a further restriction on the admissible $2$-forms (which gives a submodule $\Sigma^{N+1}$), and the process is repeated from there. If $\Sigma'=\{0\}$, then $\Sigma$ contains no maximal rank $2$-forms.

\textbf{Step 2:} Let $\{\omega^1,\dots,\omega^d\}$ be a basis of $\Sigma'$. Then there exist $\xi_a^b\in \Lambda^1(A)$ such that
\begin{equation*}
    \delta \omega^a=\xi_a^b\wedge \omega^b, \quad  a=1,\dots,d.
\end{equation*}
Since we are looking for $\omega\in \Sigma'$ such that $\delta \omega=0$, we need to find all possible $d$-tuples of $1$-forms $(\eta_1^A,\dots,\eta_d^A)$ satisfying
\begin{equation*}
    \eta_a^A\wedge \omega^a=0.
\end{equation*}
Denote the dimension of such $d$-tuples by $e$. Then $\omega = {\overline s}_a \omega^a$ is closed if and only if there are functions ${\overline s}_a(t,w)$ and ${\overline p}_A(t,w)$ satisfying the system of partial differential equations
\begin{equation}\label{deltas}
    \delta{\overline s}_a+{\overline s}_b \xi_a^b+{\overline p}_A \eta_a^A=0,
\end{equation}
where $\delta{\overline s}_a = \rho^i_\alpha\frac{\partial {\overline s}_a}{\partial x^i}e^\alpha$, where $\rho^i_\alpha$ is short for the components of the anchor map of the IP Lie algebroid.

Here, the manifold $M$ has coordinates $(x^i)=(t,w^i)$. Now, we define the extended manifold $\overline{M}=M\times \mathbb{R}^d\times \mathbb{R}^e$ as the manifold with independent coordinates $(t,w^i,s_a,p_A)$, and we consider the fiber bundle $p\colon \overline M \to M$. Since $\overline M$ is a product manifold, the $p$-prolongation of the IP Lie algebroid $A\to M$ is the Lie algebroid $\overline{A}=A\times \mathrm{T}\mathbb{R}^d\times \mathrm{T}\mathbb{R}^e$ over $\overline{M}$.

If $\{e_\alpha\}$ is a basis of $\Gamma(A)$, then in correspondence with the notation of \Cref{sec:pullback}, we may introduce a basis $\{{\bar e}_\alpha, {\overline E}_a, {\overline E}_A\}$ for $\Gamma(\overline A)$ and denote the basis of sections of the dual bundle of $\overline{A}$ by $\{{\bar e}^\alpha, {\overline E}^a, {\overline E}^A\}$. We saw before that we may write ${\overline E}^a = \delta s_a$ and ${\overline E}^A = \delta p_A$. Moreover, $\bar{e}^{\alpha}$ can be identified with $e^\alpha$, as one-forms on $\overline A$. We may now introduce forms
\begin{equation*}
    \sigma_a\coloneqq \delta s_a+s_b \xi_a^b+p_A \eta_a^A\in \Lambda^1(\overline{A})
\end{equation*}
on the prolongation. Again, the forms $\xi_a^b, \eta_a^A$ on $A$ are interpreted here as forms on $\overline A$.

For a section $i\colon (t,w)\mapsto (t,w, {\overline s}_a(t,w), {\overline p}_A(t,w))$ of $p$ (with its corresponding $i$-induced section $I$), we get
\begin{equation*}
    I^*\sigma_a = I^*\delta s_a+{\overline s}_b I^*\xi_a^b+{\overline p}_A I^*\eta_a^A =   \delta {\overline s}_a+{\overline s}_b \xi_a^b+{\overline p}_A \eta_a^A  \in \Lambda^1(A).
\end{equation*}
It is now easy to see that the problem of finding a solution of \Cref{deltas} is equivalent to finding integral manifolds of $\mathcal{I}=\langle \sigma_a\rangle$. 


We conclude:
\begin{theorem} \label{prop:new}
Let $\Sigma'$ be the maximal subspace of $\Sigma$ that is a differential ideal. Then the problem of finding closed forms within $\Sigma'$ is equivalent to finding integral manifolds of the differential ideal $\mathcal{I}=\langle \sigma_a\rangle$  on the $p$-prolongation of the IP Lie algebroid by $p\colon M\times \mathbb{R}^d\times \mathbb{R}^e\to M$. 
\end{theorem}

In one of the next subsections, we will consider a special case, where $\Sigma'=\Sigma\coloneqq \spn\{\omega^a= \omega^{ij}\colon 1\leq i,j\leq n\}$, with
\begin{equation*}
    \omega^{ij}\coloneqq \frac{1}{2}\left(\Psi^i\wedge \Theta^j+\Psi^j\wedge \Theta^i\right).
\end{equation*}

\subsubsection{\texorpdfstring{Computing the exterior derivatives of $\Psi^i$ and $\Theta^j$}{Computing the exterior derivatives of Psii and Thetaj}}

We first do some preparatory computations before applying the algorithm. To this end, we compute the exterior derivatives of $\Psi^i$ and $\Theta^j$ in terms of the basis $\{\Gamma^0,\Theta^i,\Psi^i\}$. We use the following identity for $\alpha\in \Lambda^1(A)$:
\begin{equation*}
    \delta \alpha(X,Y)=\rho(X)(\alpha(Y))-\rho(Y)(\alpha(X))-\alpha(\llbracket X,Y\rrbracket),
\end{equation*}
where $X$ and $Y$ are sections of $A\to M$. Since
\begin{multline*}
    \delta \alpha=\sum_k \delta \alpha(\Gamma_0,H_k)\Gamma^0\wedge \Theta^k+\sum_k \delta \alpha(\Gamma_0,W_k)\Gamma^0\wedge \Psi^k+\sum_{k<l} \delta\alpha(H_k,H_l)\Theta^k\wedge \Theta^l\\+\sum_{k,l} \delta \alpha(H_k,W_l)\Theta^k\wedge \Psi^l+\sum_{k<l} \delta \alpha(W_k,W_l)\Psi^k\wedge \Psi^l,
\end{multline*}
it suffices to compute $\delta \alpha(X,Y)$ for $X,Y\in \{\Gamma_0,H_i,W_i\}$. We have that
\begin{equation*}
    \delta \Psi^i(\Gamma_0,H_k)=-\phi_k^i,\quad \delta\Psi^i(\Gamma_0,W_k)=-\lambda_k^i,\quad \delta \Psi^i(H_k,H_l)=-r_{kl}^i
\end{equation*}
\begin{equation*}
    \delta\Psi^i(H_k,W_l)=-(C_{kl}^i+\lambda_{kl}^i),\quad \delta\Psi^i(W_k,W_l)=0.
\end{equation*}
Therefore,
\begin{equation}\label{eqn:dpsii}
    \delta \Psi^i=-\phi_k^i \Gamma^0\wedge \Theta^k-\lambda_k^i \Gamma^0\wedge \Psi^k-\frac{1}{2}r_{kl}^i \Theta^k\wedge \Theta^l-(C_{kl}^i+\lambda_{kl}^i) \Theta^k\wedge \Psi^l,
\end{equation}
where we have used skew-symmetry of $r_{kl}^i$ in the indices $kl$. Similarly, since
\begin{equation*}
    \delta \Theta^j(\Gamma_0,H_k)=-\lambda_k^j,\quad \delta \Theta^j(\Gamma_0,W_k)=\delta_k^j,\quad \delta\Theta^j(H_k,H_l)=-C_{kl}^j,
\end{equation*}
\begin{equation*}
    \delta \Theta^j(H_k,W_l)=0,\quad \delta \Theta^j(W_k,W_l)=0,
\end{equation*}
it follows that
\begin{equation}\label{eqn:dthetaj}
    \delta \Theta^j=-\lambda_k^j \Gamma^0\wedge \Theta^k+\Gamma^0\wedge \Psi^j-\frac{1}{2}C_{kl}^j \Theta^k\wedge \Theta^l.
\end{equation}

\subsubsection{\texorpdfstring{The EDS algorithm when $\langle \Sigma^0\rangle_{\mathrm{alg}}$ is a differential ideal}{The EDS algorithm when Sigma0 is a differential ideal}}

Let $\Sigma^0\coloneqq \spn\{\omega^{ij}\}$ and $\langle \Sigma^0\rangle_{\mathrm{alg}}$ be the algebraic ideal generated by the $\omega^{ij}$, where
\begin{equation*}
    \omega^{ij}\coloneqq \frac{1}{2}\left(\Psi^i\wedge \Theta^j+\Psi^j\wedge \Theta^i\right).
\end{equation*}
Note that $\omega^{ij}=\omega^{ji}$. We first determine when $\langle \Sigma^0\rangle_{\mathrm{alg}}$ is a differential ideal. Let $\omega$ be a $2$-form in $\Sigma^0$. Then, it can be written as $\omega=s_{ij}\omega^{ij}$, for some unique symmetric matrix $(s_{ij})$. We now compute the exterior derivative of $\omega$. We see that
\begin{equation*}
    \delta \omega=\delta s_{ij}\wedge \omega^{ij}+s_{ij}\left[\delta \Psi^i\wedge \Theta^j-\Psi^i\wedge \delta \Theta^j\right].
\end{equation*}
Using \Cref{eqn:dpsii,eqn:dthetaj}, one obtains
\begin{multline*}
    \delta \omega=\delta s_{ij}\wedge \omega^{ij}+s_{ij}\Bigg[\phi_k^i \Gamma^0\wedge \Theta^j\wedge \Theta^k-\lambda_k^j \Gamma^0\wedge \Psi^i\wedge \Theta^k-\lambda_k^i \Gamma^0\wedge \Psi^k\wedge \Theta^j\\+\frac{1}{2}C_{kl}^j \Psi^i\wedge \Theta^k\wedge \Theta^l+\left(-\frac{1}{2}C_{kl}^i+\frac{1}{2}\frac{\partial^2 \gamma^i}{\partial w^k \partial w^l}\right)\Psi^l\wedge \Theta^j\wedge \Theta^k-\frac{1}{2}r_{kl}^i\Theta^j\wedge \Theta^k\wedge \Theta^l\Bigg],
\end{multline*}
where we have used that $s_{ij} \Gamma^0\wedge \Psi^i\wedge \Psi^j=0$, and $C_{kl}^i+\lambda_{kl}^i=\frac{1}{2}C_{kl}^i-\frac{1}{2}\frac{\partial^2 \gamma^i}{\partial w^k \partial w^l}$. Since
\begin{equation*}
    -s_{ij}\lambda_k^i \Gamma^0\wedge \Psi^k\wedge \Theta^j-s_{ij}\lambda_k^j \Gamma^0\wedge \Psi^i\wedge \Theta^k=-2s_{ij}\lambda_k^i \Gamma^0\wedge \omega^{kj}=-2s_{kj}\lambda_i^k \Gamma^0\wedge \omega^{ij},
\end{equation*}
\begin{equation*}
    -\frac{s_{ij}}{2}C_{kl}^i \Psi^l\wedge \Theta^j\wedge \Theta^k+\frac{s_{ij}}{2}C_{kl}^j \Psi^i\wedge \Theta^k\wedge \Theta^l=s_{ij}C_{kl}^i \Theta^l\wedge \omega^{kj}=s_{kj}C_{il}^k \Theta^l\wedge \omega^{ij},
\end{equation*}
and
\begin{equation*}
    \frac{s_{ij}}{2}\frac{\partial^2 \gamma^i}{\partial w^k \partial w^l} \Psi^l\wedge \Theta^j\wedge \Theta^k=-\frac{s_{ij}}{2}\frac{\partial^2 \gamma^i}{\partial w^k \partial w^l} \Theta^j\wedge \omega^{kl}=-\frac{s_{kl}}{2}\frac{\partial^2 \gamma^k}{\partial w^i \partial w^j}\Theta^l\wedge \omega^{ij},
\end{equation*}
it follows that
\begin{equation}\label{eqn:deltaomega}
    \delta \omega=\left[\delta s_{ij}-2s_{kj}\lambda_i^k \Gamma^0+s_{kj}C_{il}^k \Theta^l-\frac{s_{kl}}{2}\frac{\partial^2 \gamma^k}{\partial w^i \partial w^j}\Theta^l\right]\wedge \omega^{ij}+s_{ik}\phi_j^k \Gamma^0\wedge \Theta^i\wedge \Theta^j-\frac{s_{il}}{2}r_{jk}^l \Theta^i\wedge \Theta^j\wedge \Theta^k.
\end{equation}
For $\delta \omega\in \langle \Sigma^0\rangle_{\mathrm{alg}}$ to hold, the last two sets of terms in \Cref{eqn:deltaomega} must vanish. The first vanishes if and only if
\begin{equation}\label{eqn:phicondition}
    s_{ik}\phi_j^k=s_{jk}\phi_i^k,
\end{equation}
and the second vanishes if and only if $s_{ij}$ satisfies
\begin{equation}\label{eqn:curvaturecondition}
    s_{il}r_{jk}^l+s_{jl}r_{ki}^l+s_{kl}r_{ij}^l=0.
\end{equation}
Note that the first equation is the Helmholtz condition generated by $\delta \Omega(\Gamma_0,H_i,H_j)=0$ (see also \Cref{def:reducedmultiplier}), and the second is the redundant Helmholtz condition \eqref{eqn:redundanthelmholtz2} (obtained via $\delta \Omega(H_i,H_j,H_k)=0$).

We make an extra assumption, for the rest of the paper: we suppose that \Cref{eqn:phicondition,eqn:curvaturecondition} hold identically for every symmetric matrix $(s_{ij})$. This assumption is inspired by the fact that this occurs in the  examples that we will consider later and it allows us to skip Step 1 of the EDS algorithm. In that case,
\begin{equation*}
    \delta\omega=\left[\delta s_{ij}-2s_{kj}\lambda_i^k \Gamma^0+s_{kj}C_{il}^k \Theta^l-\frac{s_{kl}}{2}\frac{\partial^2 \gamma^k}{\partial w^i \partial w^j}\Theta^l\right]\wedge \omega^{ij}.
\end{equation*}
It can be verified that if $\eta_{ij}=\eta_{ji}$ is a collection of $1$-forms such that
\begin{equation*}
    \eta_{ij}\wedge \omega^{ij}=0,
\end{equation*}
then $\eta_{ij}$ is of the form
\begin{equation*}
    \eta_{ij}=P_{ijk}\Psi^k+Q_{ijk}\Theta^k,
\end{equation*}
where $P_{ijk}$ and $Q_{ijk}$ are symmetric in the indices $ijk$. Hence $\omega=s_{ij}\omega^{ij}$ is closed if and only if $s_{ij}$ are solutions to the system $\sigma_{ij}=0$, where
\begin{equation*}
    \sigma_{ij}\coloneqq \delta s_{ij}-(s_{ki}\lambda_j^k+s_{kj}\lambda_i^k)\Gamma^0+P_{ijl}\Psi^l+\left[\frac{1}{2}\left(s_{ki}C_{jl}^k+s_{kj}C_{il}^k-s_{kl}\frac{\partial^2 \gamma^k}{\partial w^i \partial w^j}\right)+Q_{ijl}\right]\Theta^l.
\end{equation*}
In the current set-up, the extended manifold is $\overline{M}=M\times \mathbb{R}^{n^2}\times \mathbb{R}^{2n^{3}}$, with coordinates
\begin{equation*}
    (t,w^i,s_{ij},P_{ijk},Q_{ijk}).
\end{equation*}
As before in \Cref{introsubsec}, via our construction in \Cref{sec:pullback}, we obtain the $p$-prolongation Lie algebroid $\bar{\tau}\colon \overline{A}\to \overline{M}$, where a basis of sections of the dual bundle of $\overline{A}=A\times \mathrm{T}\mathbb{R}^{n^2}\times \mathrm{T}\mathbb{R}^{2n^{3}}$ is given by $\{\Gamma^0,\Psi^i,\Theta^i,\delta s_{ij},\delta P_{ijk},\delta Q_{ijk}\}$. We can then regard $\sigma_{ij}$ as $1$-forms on $\overline{A}$.

We are finally ready to state a special case  of \Cref{prop:new}.

\begin{theorem}\label{thm:helmholtzintmfd}
  Suppose that $\langle \Sigma^0\rangle_{\mathrm{alg}}$ is a differential ideal. Let $\mathcal{I}=\langle \sigma_{ij}\rangle$. A solution $(k_{ij})$ to the reduced Helmholtz conditions exists if and only if there exists an integral manifold $i\colon M\to \overline{M}$,
\begin{equation*}
    (t,w)\longmapsto \left(t,w,s_{ij}=\bar{s}_{ij}(t,w),P_{ijk}=\overline{P}_{ijk}(t,w),Q_{ijk}=\overline{Q}_{ijk}(t,w)\right)
\end{equation*}
of $\mathcal{I}$ on  the IP Lie algebroid, with $\det(\bar{s}_{ij})\neq 0$.
\end{theorem}
\begin{proof}
We first do some preparatory computations. Note that
\begin{equation*}
    I^{\ast}(s_{ij})=\gamma(\bar{s}_{ij})+\frac{\partial \bar{s}_{ij}}{\partial w^l}\Psi^l+\psi_l^k \frac{\partial \bar{s}_{ij}}{\partial w^k}\Theta^l.
\end{equation*}
Hence
\begin{multline*}
    I^{\ast}(\sigma_{ij})=\left[\gamma(\bar{s}_{ij})-\bar{s}_{ki}\lambda_j^k-\bar{s}_{kj}\lambda_i^k\right]\Gamma^0+\left[\overline{P}_{ijl}+\frac{\partial \bar{s}_{ij}}{\partial w^l}\right]\Psi^l\\+\left[\psi_l^k \frac{\partial \bar{s}_{ij}}{\partial w^k}+\frac{1}{2}\left(\bar{s}_{ki}C_{jl}^k+\bar{s}_{kj}C_{il}^k-\bar{s}_{kl}\frac{\partial^2 \gamma^k}{\partial w^i \partial w^j}\right)+\overline{Q}_{ijl}\right]\Theta^l.
\end{multline*}
Note that $I^{\ast}(\sigma_{ij})=0$  if and only if the following equations hold:
\begin{enumerate}[label=(\alph*)]
    \item $\gamma(\bar{s}_{ij})-\bar{s}_{ki}\lambda_j^k-\bar{s}_{kj}\lambda_i^k=0$,
    \item $\overline{P}_{ijl}+\frac{\partial \bar{s}_{ij}}{\partial w^l}=0$ for all $l\in \{1,\dots,n\}$,
    \item $\psi_l^k \frac{\partial \bar{s}_{ij}}{\partial w^k}+\frac{1}{2}\left(\bar{s}_{ki}C_{jl}^k+\bar{s}_{kj}C_{il}^k-\bar{s}_{kl}\frac{\partial^2 \gamma^k}{\partial w^i \partial w^j}\right)+\overline{Q}_{ijl}=0$ for all $l\in \{1,\dots,n\}$.
\end{enumerate}
Suppose first that a solution to the reduced Helmholtz conditions exists. Denote this solution by $(k_{ij})=(\bar{s}_{ij})$. Then it follows directly that (a) holds. We now claim that
\begin{equation*}
    \overline{P}_{ijl}=-\frac{\partial \bar{s}_{ij}}{\partial w^l},\quad \overline{Q}_{ijl}=-\psi_l^k \frac{\partial \bar{s}_{ij}}{\partial w^k}-\frac{1}{2}\left(\bar{s}_{ki}C_{jl}^k+\bar{s}_{kj}C_{il}^k-\bar{s}_{kl}\frac{\partial^2 \gamma^k}{\partial w^i \partial w^j}\right),
\end{equation*}
is a solution to equations (b) and (c). The only requirement that these functions must satisfy is that $\overline{P}_{ijl}$ and $\overline{Q}_{ijl}$ be totally symmetric. By construction, they are both symmetric in $ij$. Furthermore, note that
\begin{equation*}
    \overline{P}_{ijl}=-\frac{\partial \bar{s}_{ij}}{\partial w^l}=-\frac{\partial \bar{s}_{lj}}{\partial w^i}=\overline{P}_{lji}.
\end{equation*}
As such, $\overline{P}_{ijl}$ is also symmetric in $il$. Hence $\overline{P}_{ijl}$ is totally symmetric. Similarly, $\overline{Q}_{ijl}$ is symmetric in the indices $jl$, because by one of the redundant Helmholtz conditions (specifically, \Cref{eqn:redundanthelmholtz1}), and the symmetry of $\bar{s}_{ij}$, one gets
\begin{equation*}
    \overline{Q}_{ijl}=-\psi_j^k \frac{\partial \bar{s}_{il}}{\partial w^k}-\frac{1}{2}\left(\bar{s}_{ki}C_{lj}^k+\bar{s}_{kl}C_{ij}^k-\bar{s}_{kj}\frac{\partial^2 \gamma^k}{\partial w^i \partial w^l}\right)=\overline{Q}_{ilj},
\end{equation*}
which implies that $\overline{Q}_{ijl}$ is totally symmetric.

Conversely, suppose that we are given a solution $(\bar{s}_{ij},\overline{P}_{ijk}, \overline{Q}_{ijk})$ to the system of equations consisting of (a), (b), (c) and $\det(\bar{s}_{ij})\neq 0$. We claim that $(\bar{s}_{ij})$ satisfies the reduced Helmholtz conditions. Note that if $\langle \Sigma^0\rangle_{\mathrm{alg}}$ is a differential ideal, the condition
\begin{equation*}
    \bar{s}_{ki}\phi_j^k=\bar{s}_{kj}\phi_i^k
\end{equation*}
holds by assumption. It therefore remains to prove that
\begin{equation}\label{eqn:helmholtzDV}
    \frac{\partial \bar{s}_{jk}}{\partial w^i}=\frac{\partial \bar{s}_{ik}}{\partial w^j}.
\end{equation}
But, by equation (b), we obtain
\begin{equation*}
    \overline{P}_{jki}+\frac{\partial \bar{s}_{jk}}{\partial w^i}=0,
\end{equation*}
and
\begin{equation*}
    \overline{P}_{ikj}+\frac{\partial \bar{s}_{ik}}{\partial w^j}=0.
\end{equation*}
Since $\overline{P}_{ijk}$ is totally symmetric, it follows that $\overline{P}_{ikj}=\overline{P}_{jki}$. Hence, $\bar{s}_{ij}$ must satisfy \Cref{eqn:helmholtzDV}, as required.
\end{proof}

\begin{example}
The canonical connection on a Lie group is defined via the covariant derivative operator $\nabla_X Y=\frac{1}{2}[X,Y]$, where $X$ and $Y$ are any left-invariant vector fields on $G$. The invariant inverse problem for the canonical connection was studied in \cite{crampin2008invariantinv,muzsnay2005invariantinv}. In this case, the reduced vector field is given by $\gamma=\frac{\partial}{\partial t}$, so that $\gamma^i=0$ for all $i\in \{1,\dots,n\}$. We then have that
\begin{equation*}
    \lambda_i^j=\frac{1}{2}w^k C_{ki}^j=\psi_i^j,
\end{equation*}
\begin{equation*}
    \phi_i^j=-\lambda_i^k \lambda_k^j=-\frac{1}{4}w^l w^m C_{li}^k C_{mk}^j,
\end{equation*}
\begin{equation*}
    \lambda_{ij}^k=-\frac{1}{2}C_{ij}^k,
\end{equation*}
\begin{equation*}
    r_{ij}^k=\frac{3}{4}w^m(C_{mi}^l C_{jl}^k-C_{mj}^l C_{il}^k)+\frac{1}{2}w^m C_{ml}^k C_{ij}^l.
\end{equation*}
We first consider a situation where a reduced multiplier exists. If $G$ is the one-dimensional Lie group ${\mathbb R}^1$ of the real line, then $C_{11}^1=0$, $\phi_1^1=0$ and $r_{11}^1=0$. Then $\langle \Sigma^0\rangle_{\mathrm{alg}}$ is a differential ideal. In this case, one easily verifies that a solution to the reduced Helmholtz conditions is given by $k_{11}(t,w^1)=g(w^1)$, where $g\neq 0$ is any nowhere zero smooth function. On the other hand, if one was to use integral manifolds on Lie algebroids, note that
\begin{equation*}
    I^{\ast}(\sigma_{11})=\gamma(\bar{s}_{11})\Gamma^0+\left[\overline{P}_{111}+\frac{\partial \bar{s}_{11}}{\partial w^1}\right]\Psi^1+\overline{Q}_{111}\Theta^1.
\end{equation*}
Hence, a section $i\colon M\to \overline{M}$ satisfying $I^{\ast}(\sigma_{11})=0$ is given by $\bar{s}_{11}(t,w^1)=g(w^1)$, $\overline{P}_{111}(t,w^1)=-\frac{\partial g(w^1)}{\partial w^1}$ and $\overline{Q}_{111}(t,w^1)=0$. Since \Cref{thm:helmholtzintmfd} asks for $\bar{s}_{11}\neq 0$, this is in agreement with our result. 

Next, we consider the Heisenberg group $G$. Its Lie algebra is $3$-dimensional, with the single non-trivial bracket $\{E_1,E_3\}=E_2$. Hence,
\begin{equation*}
    C_{ij}^k=(\delta_{i1}\delta_{j3}-\delta_{i3}\delta_{j1})\delta^{k2}.
\end{equation*}
It was shown in \cite{crampin2008invariantinv} that an autonomous reduced multiplier does not exist (due to the condition $\det(k_{ij})\neq 0$). It is not difficult to see that we may extend this observation to the time-dependent case.

Here, it corresponds to a situation where $\langle \Sigma^0\rangle_{\mathrm{alg}}$ is a differential ideal, since $\phi_i^j=0$ and $r_{ij}^k=0$. Since there does not exist a reduced multiplier, as a consequence of \Cref{thm:helmholtzintmfd}, one can also not find a section $i$ such that $I^{\ast}(\sigma_{ij})=0$ and $\det(\bar{s}_{ij})\neq 0$.
\end{example}

\section{Outlook}\label{sec:outlook}
We end the paper with two different directions for future research. 

First, it may be possible to extend our discussion on the invariant inverse problem, as follows. Lagrangian mechanics, in the generalization that the configuration bundle itself is an arbitrary Lie algebroid, has been studied in many papers, beginning with those by A.\ Weinstein \cite{weinstein1996lagrangian} and E.\ Mart\'inez \cite{martinez2001lagrangeliealgebroid}. Studies of the inverse problem for Lagrangians on general Lie
algebroids were done in e.g., \cite{barbero2016invproblemliealgebroids,  popescu2011metric, popescu2014invproblemliealgebroids}. In particular, this is relevant for the  \emph{autonomous} inverse problem: the original inverse problem corresponds to Lagrangians on a tangent bundle $\mathrm{T}M$.  Besides, the EDS approach to the original inverse problem typically makes use of  \emph{time-dependent} Lagrangians on $\mathbb R\times\mathrm{T}M$. Likewise, our current extension of the invariant inverse problem deals with time-dependent Lagrangians on $\mathbb{R}\times \mathfrak{g}$. In the literature, one may also find many papers on the extension of Lagrangian mechanics on Lie algebroids to the time-dependent case,  using affine bundles (starting with  \cite{martinez2002prolongation}). 

Taking all this into account, and considering the similarity between Thm.\ 4.8 of \cite{barbero2016invproblemliealgebroids} and \Cref{thm:twoform} in our paper, we think that it is very likely that we can also use our EDS approach to determine conditions on the existence of solutions to the Helmholtz conditions for a general Lie algebroid. The invariant inverse problem on a Lie group we have discussed in this paper could then be considered as a special case.

In the first half of the paper we have introduced the notion of an integral manifold, for an EDS on a Lie algebroid. In an upcoming paper, we will develop the theory of such integral manifolds in more depth and we will outline some methods for establishing the existence of integral manifolds. Specifically, we will extend the notion of an integral element. This will allow us to prove a Lie algebroid version of the acclaimed Cartan--K{\"a}hler theorem. This theorem, which appears in many seemingly different versions in the literature (cf. \cite{bryant1991exterior,ivey2016cartan}), studies in essence the integrability of PDEs in a geometric fashion. It relies on a process by which one may successively enlarge integral manifolds of a certain dimension to integral manifolds of one higher dimension. 

\paragraph{Acknowledgements.} We are indepted to Sonja Hohloch for many useful discussions. The authors thank the Research Fund of the University of Antwerp (BOF) for its support through the DOCPRO project 46954.

\end{document}